\documentclass[reqno, 12pt]{amsart}

\usepackage[utf8]{inputenc}

 \usepackage{tikz-cd}

\usepackage{amsfonts, amsthm, amsmath, amssymb}
\usepackage{hyperref}
\hypersetup{colorlinks=false}
\usepackage[all]{xy}

\usepackage{pifont}

 \usepackage[margin=1in]{geometry}

\usepackage[colorinlistoftodos,french,bordercolor=black,backgroundcolor=red,linecolor=red,textsize=footnotesize,textwidth=0.75in,shadow]{todonotes}

\usepackage{hyperref}
\usepackage{dsfont}
\usepackage{upgreek}
\usepackage{mathabx,yfonts}
\usepackage{enumerate}

\numberwithin{equation}{section}

\newtheorem{theorem}{Theorem}[section] 
\newtheorem{lemma}[theorem]{Lemma}

\newtheorem{proposition}[theorem]{Proposition}

\newtheorem{corollary}[theorem]{Corollary}

\theoremstyle{definition}

\newtheorem{remark}[theorem]{Remark}
\newtheorem{definition}[theorem]{Definition}
\newtheorem{example}[theorem]{Example}

\renewcommand{\phi}{\varphi}

\newcommand{\PP}{\mathbb{P}}
\newcommand{\FF}{\mathbb{F}}
\newcommand{\ZZ}{\mathbb{Z}}

\newcommand{\NN}{\mathbb{N}}
\newcommand{\QQ}{\mathbb{Q}}
\newcommand{\RR}{\mathbb{R}}
\newcommand{\CC}{\mathbb{C}}

\newcommand{\E}{\mathbb{E}}

\renewcommand{\leq}{\leqslant}
\renewcommand{\le}{\leqslant}
\renewcommand{\geq}{\geqslant}
\renewcommand{\ge}{\geqslant}
\renewcommand{\bar}{\overline}

\newcommand{\x}{\mathbf{x}}

\renewcommand{\c}{\mathbf{c}}

\renewcommand{\b}{\mathbf{b}}

\renewcommand{\r}{\mathbf{r}}

\newcommand{\fp}{\mathfrak{p}}

\DeclareMathOperator{\rk}{rk}

 \DeclareMathOperator{\disc}{disc}
\DeclareMathOperator{\Li}{Li}

\DeclareMathOperator{\Pic}{Pic}
\DeclareMathOperator{\Div}{Div}

\DeclareMathOperator{\Spec}{Spec}

\DeclareMathOperator{\Var}{Var} 
\DeclareMathOperator{\Cov}{Cov}

\DeclareMathOperator{\SL}{SL}
\DeclareMathOperator{\GL}{GL}

\let\emptyset\varnothing

\DeclareSymbolFont{bbold}{U}{bbold}{m}{n}
\DeclareSymbolFontAlphabet{\mathbbold}{bbold}

\renewcommand{\P}{\mathbb{P}}
\newcommand{\Q}{\mathbb{Q}} 

\newcommand{\F}{\mathbb{F}}
\newcommand{\N}{\mathbb{N}}
\newcommand{\R}{\mathbb{R}}

\newcommand{\Z}{\mathbb{Z}}
\newcommand{\A}{\mathbb{A}}
\renewcommand{\l}{\left}

\renewcommand{\r}{\right}
\renewcommand{\b}{\mathbf}
\renewcommand{\c}{\mathcal}
\renewcommand{\epsilon}{\varepsilon}

\renewcommand{\leq}{\leqslant}
\renewcommand{\geq}{\geqslant}

\DeclareMathOperator*{\Osum}{\sum{}^*}

\title{Multivariate normal distribution for integral
 points  on varieties}

\usepackage{color}

\author{Daniel El-Baz}
\address{
Graz University of Technology\\
Institute of Analysis and Number Theory\\
Steyrergasse 30/II, 8010 Graz, Austria}
\email{el-baz@math.tugraz.at}

\author{Daniel Loughran} 
  \address{
  Department of Mathematical Sciences \\
University of Bath \\
Claverton Down \\
Bath \\
BA2 7AY \\
UK}
\urladdr{https://sites.google.com/site/danielloughran}

\author{Efthymios Sofos} 
\address{
Department of Mathematics\\
University of Glasgow,
University Place,
 Glasgow,  G12 8QQ, UK}
\email{efthymios.sofos@glasgow.ac.uk}

 \subjclass[2010]
{14G05; %Rational points
60F05, %Central limit and other weak theorems
11N36. %Applications of sieve methods
}
\date{\today}

\newcommand{\beq}[2]
{
\begin{equation}
\label{#1}
{#2}
\end{equation}
}

\begin{document}

%\vspace{-0,5cm}

\begin{abstract}
	Given a variety with coefficients in $\Z$, we study the distribution of the number of 
	primes dividing the coordinates as we vary an integral point. 
	Under suitable assumptions, we show that this has a 
	multivariate normal distribution. 
	We generalise this to more general Weil divisors, where we obtain
	a geometric interpretation of the covariance matrix.
	For our results we develop a version of the Erd\H os--Kac 
	theorem that applies to fairly general integer sequences and
	does not require a positive exponent of level of distribution.
\end{abstract}

\maketitle

\vspace{-1cm}

\setcounter{tocdepth}{1}
\tableofcontents

\vspace{-1,5cm}

\section{Introduction}
\label{introduction}
\subsection{Erd\H{o}s--Kac} \label{sec:EK_intro}

To study the prime factorisation of a non-zero integer $m$, Erd\H{o}s and Kac~\cite{EK40}
 considered the distribution of the function
\vspace{-0,1cm}
$$\omega(m ) = \mbox{number of distinct primes $p$ such that $p $ divides $m$}.$$
They showed  that $\omega(m)$  behaves like a normal 
distribution with mean $\log \log m$ and variance $\log \log m$.
More precisely, let
$\Omega_B = \{m\in \N: m\leq B \} $
be
equipped with the  uniform probability 
measure for $B\geq 1 $. Then as $B \to \infty$ the sequence of random variables
$$
\Omega_B \to \RR, \quad
m \mapsto \frac{\omega(m) - \log \log B}{\sqrt{\log \log B}}
$$
converges in distribution to the
normal distribution 
with mean
$0$ and variance $1$. 
Their work
is a foundational result in  
probabilistic number theory and opened up many new research directions;  we refer to the paper \cite{grs} and the references therein for various generalisations. 

In our paper we study prime divisors of integers in sparse sequences, with an emphasis on solutions to Diophantine equations.
A very special case of our results is as follows.
\begin{theorem}
\label{thm:sexycase}
Let  $f \in \Z[x_1,\ldots,x_n]$
be  a non-singular homogeneous polynomial
 with $n > (\deg(f) -1) 2 ^{\deg(f) } $. Let
$\Omega_B
\! = \!
 \{\x\in \Z^n: 
 f(\b x )=0, \max_i |x_i |  \leq B, \gcd(x_1,\dots,x_n) = 1 \}$ be
equipped with the uniform probability measure.
If $f(\x)= 0$ has a non-trivial integer solution, then as $B\to 
\infty$ the
random vectors
 	$$
\Omega_B \to \R^n, 
\quad
\b x=(x_1,\ldots , x_n ) 
 \mapsto 
\left(\frac{\omega(  x_1  ) - \log \log B}{\sqrt{\log \log B}}, \ldots, \frac{\omega( x_n   ) - \log \log B}{\sqrt{\log \log B}} \right)$$
	converge in distribution to the standard multivariate normal distribution 
	on $\R^n$.
\end{theorem}

By the \emph{standard multivariate normal distribution}, we mean the multivariate normal distribution with zero mean vector and identity covariance matrix. 
We refer the reader to \S\ref{sec:conventions} for a reminder on multivariate normal distributions.

One knows how to count the number of solutions to
 the equation $f(\x) = 0$ using the circle method. 
Our motivation comes from trying to understand the more subtle arithmetic properties of the solutions, and is partly motivated by Sarnak's saturation problem \cite{BGS10}, which asks whether there are solutions with 
coordinates being
 prime or 
 almost prime.

Theorem \ref{thm:sexycase} shows that the coordinate $x_i$ typically has $\log \log |x_i|$ prime factors. Moreover, it compares the numbers of prime factors of different coordinates; the fact that we obtain the identity covariance matrix means that the number of prime factors of different coordinates is `uncorrelated', something which is not a priori
obvious. We  have a purely geometric interpretation of this phenomenon, which we explain in more detail later  (Theorem \ref{thm:main}).

We are only aware of a few papers in the literature in probabilistic number theory which deal with a multivariate distribution: LeVeque \cite[\S4]{Lev49} 
on $(\omega(m),\omega(m+1))$
(stated by Erd\H{o}s without proof \cite{Erd46}), 
Halberstam \cite{Hal56}
again, on $(\omega(m),\omega(m+1))$,
and Tanaka \cite{Tan55}
on the distribution of 
$(\omega(f_1(m)), \ldots, \omega(f_n(m))$, where $f_i$ are  restricted to be 
pairwise coprime integer  univariate
polynomials. These can all be obtained as special cases of our most general result on a multivariate version of the Erd\H{o}s--Kac theorem for integer sequences satisfying certain hypotheses (see \S\ref{sec:EK}, in particular Theorem \ref{thm:mainthmr}). This more general result allows one to 
prove a general version of Tanaka's result, with no restrictions on $f_i$ and, furthermore, to
replace $\omega$ by any strongly additive function in Theorem \ref{thm:sexycase}. It
may be viewed as a multidimensional version of Billingsley's work~\cite[\S 3]{MR0466055}.

\subsection{Distribution of the prime divisors of the coordinates} \label{sec:coordinates}
Let $X \subset \PP^{n-1}_\QQ$ be a projective variety over $\QQ$. 
For 
 $x  
\in \PP^{n-1} (\QQ)$ 
 we choose a representative $\b x \in \Z^n$ with $\gcd(x_1,\ldots, x_n)=1$
such that $x=(x_1:\dots:x_n) $. Recall that the naive
 height of $x$ is defined through 
$H(x) = \max \{|x_1|,\dots,|x_n|\}.$
We are interested in the distribution of  $\omega(x_i)$, which only depends on $x\in \PP^{n-1} (\QQ)$ and is well-defined providing $x_i \neq 0$.

\subsubsection{Complete intersections}

For $R \ge 1$ and $1\leq i \leq R $, let $f_i \in \Z[X_1,\ldots,X_n]$ be homogeneous of the same degree $D$. 
The Birch rank, denoted by $\mathfrak{B}(\b  f)$,
is defined to be the codimension of the affine variety in 
$\mathbb{C}^n$  given by 
\begin{equation}
\rk \left(\left(\frac{\partial f_i(\b x)}{\partial x_j}\right)_{\substack{1\leq i\leq R,1\leq j\leq n}}\right)<R.
\end{equation}

\begin{theorem} \label{thm:CI}
	Let $X \subset \PP^{n-1}_\QQ$ be the
	 complete intersection 
given by 	$f_1=\ldots=f_R=0$
	as above and let 
	$\Omega_B=\{x \in X(\Q): H(x) \leq B , x_1\cdots x_n \neq 0 \}$ be
	equipped with the
	uniform probability measure. 
	Assume that $X$ is smooth  and 
	 $\mathfrak{B}(\b f)>2^{D-1}(D-1)R(R+1)$. If $X(\Q) \neq \emptyset$ then
	as $B\to \infty$	the random vectors
	$$\Omega_B \to \R^n,  \quad
	x= (x_1:\ldots: x_n) \mapsto 
	\left(\frac{\omega(x_1) - \log \log B}{\sqrt{\log \log B}}, \ldots, \frac{\omega (x_n) - \log \log B}{\sqrt{\log \log B}} \right)$$
	converge in distribution to the standard multivariate normal distribution  on $\R^n$.
	\end{theorem}

\subsubsection{Homogeneous spaces}
Another class of examples to which our main result applies is given by certain symmetric varieties in affine space.
We defer the precise definition of this class to \S \ref{sec:symmvar} and instead present our results for two explicit families of such varieties.

Let $Q$ be a non-degenerate, indefinite integral quadratic form in $n \ge 3$ variables. For each $k \in \Z \setminus \{0\}$, we consider the variety 
\[
L_k: \qquad Q(\b  x) = k \quad \subset \A^n_\Z
\]
equipped with the usual height function $H(\x) = \max_i|x_i|$.

\begin{theorem} \label{thm:indQF}
Let $k \in  \Z \setminus \{0\}$ and $n \ge 3$. If $n=3$, assume that  $-k \operatorname{disc}(Q)$ is not a perfect square. Let $\Omega_B = \{ \x \in L_k(\Z): H(\x) \leq B, x_1\cdots x_n \neq 0\}$ be
equipped with the uniform probability measure. If $L_k(\Z) \neq \emptyset$ then
	as $B\to \infty$	the random vectors
$$\Omega_B \to \R^n,  \quad
	\x\mapsto  \left(\frac{\omega(x_1) - \log \log B}{\sqrt{\log \log B}}, \ldots, \frac{\omega(x_n) - \log \log B}{\sqrt{\log \log B}} \right)$$
converge in distribution to the standard multivariate normal distribution on $\R^n$.
\end{theorem}

For $n \ge 2$ and $k \in \Z \setminus \{ 0 \}$, consider the variety 
\[
V_{n,k}: \qquad \det(M) = k \quad \subset \A^{n^2}_\Z,
\]
where $\det$ denotes the determinant,   viewed as a homogeneous polynomial of degree $n$ (in particular, $V_{n,1} = \mathrm{SL}_n$). 

\begin{theorem} \label{thm:deteqn}
Let $k \in \Z \setminus \{ 0 \}$, $n \geq 2$
and  $\Omega_B = \{ M = (m_{i,j}) \in V_{n,k}(\Z) : H(M) \leq B, m_{i,j} \neq 0\}$
be equipped with the uniform probability measure. As $B \to \infty$ the random vectors 
$$\Omega_B \to \R^{n^2}  \quad
	M=(m_{i,j}) \mapsto \left(\frac{\omega(m_{i,j}) - \log \log B}{\sqrt{\log \log B}}\right)_{i,j \in \{1,\dots,n\}}$$
converge in distribution to the standard multivariate normal distribution on $\R^{n^2}$.
\end{theorem}

\subsubsection{Conics} \label{sec:conics}
In all the above cases, we obtained the identity covariance matrix, meaning that the random variables given by each coordinate are independent. In the case of plane conics however, we obtain a very different result. Firstly, we need to choose a different normalisation, as it turns out that $\omega(x_i)$ need not have average order $\log \log B$ in general. Secondly, there may be non-trivial correlations.

\begin{theorem} \label{thm:conics}
	Let $C \subset \PP^2_\QQ$ be a smooth plane conic with $C(\Q) \neq \emptyset$. Let 	$\Omega_B=\{x \in C(\Q): H(x) \leq B , x_1x_2x_3 \neq 0 \}$ 
	be
	equipped with the
	uniform probability measure. 
	Let $c_{i,j}$ denote the number of common irreducible components (counted without multiplicity) of 
	the divisors $x_i = 0$ and $x_j = 0$ on $C$. Then the random
	vectors
	$$\Omega_B \to \R^3,  \quad
	x \mapsto 
	\left(\frac{\omega(x_1) - c_{1,1}\log \log B}{\sqrt{c_{1,1}\log \log B}},
	\frac{\omega(x_2) - c_{2,2}\log \log B}{\sqrt{c_{2,2}\log \log B}},
	\frac{\omega(x_3) - c_{3,3}\log \log B}{\sqrt{c_{3,3}\log \log B}} \right)$$
	converge in distribution
	to a central multivariate normal distribution with
	covariance matrix whose $(i,j)$-entry is
	$(c_{i,j}/ \sqrt{c_{i,i} c_{j,j}})$.
\end{theorem}

\begin{example} \hfill
	\begin{enumerate}
	\item Take $C: x_1^2 + x_2^2 = x_3^2$. The divisors $x_1=0$ and $x_2=0$
	have two irreducible components while $x_3=0$ is irreducible, and these have
	no components in common. Hence we just
	obtain the identity matrix for the covariance
	 matrix
	$$
\begin{pmatrix}
	1 & 0 & 0 \\
	0 & 1 & 0 \\
	0 & 0 & 1 
\end{pmatrix},$$
so there are no correlations between the number of prime divisors of the coordinates.
	\item Take $C: x_1x_2 + x_2x_3 + x_3x_1 = 0$. Every divisor
	$x_i = 0$ is a union of two rational points, and they each contain
	one point in common. We obtain the covariance matrix
	$$
\begin{pmatrix}
	1 & 1/2 & 1/2 \\
	1/2 & 1 & 1/2 \\
	1/2 & 1/2 & 1 
\end{pmatrix}.$$
	This is not the identity matrix, which is reflected by the fact
	that there is a non-trivial relation between the 
	prime divisors of $x_i$ and $x_j$,
	as is clear from the equation.
	\item Take $C : x_1x_2=x_3^2$. Here $x_1=0$ and $x_2=0$ are both irreducible and are the irreducible components of  $x_3 = 0$ (we do not count irreducible components with multiplicity). The covariance
	matrix is therefore
	$$
\begin{pmatrix}
	1 & 0 & 1/\sqrt{2} \\
	0 & 1 & 1/\sqrt{2} \\
	1/\sqrt{2} & 1/\sqrt{2} & 1 
\end{pmatrix}.$$	
This matrix is singular; this means that the associated probability measure
is supported on a proper linear subspace of $\R^3$. From the equation it is also clear that the prime divisors of $x_3$ 
are completely determined by those of $x_1$ and $x_2$.
	\end{enumerate}

\end{example}
The example with singular covariance matrix is essentially the only  example for conics.

\begin{theorem} \label{thm:conic_singular}
	Let $C \subset \PP^2_\QQ$ be a smooth plane conic for which the associated covariance matrix in Theorem \ref{thm:conics} is singular. Then, up to permuting coordinates, the conic has the equation $x_1x_2 = cx_3^2$ for some $c \in \QQ$.
\end{theorem}

\subsection{A geometric reformulation}

We now come to our most general results. To state them we require some 
notation. 

Let $X \subset \PP^d_\QQ$ be a quasi-projective variety over $\QQ$. Then the usual height on projective space induces a height function $H: X(\QQ) \to \RR_{>0}$. Let $\mathcal{X}$ be a choice of model for $X$ over $\ZZ$. Then the model allows us to define the set of integral point $\mathcal{X}(\ZZ)$, which is naturally a subset of $X(\QQ) = \mathcal{X}(\QQ)$.

We assume that $\mathcal{X}$ and the height $H$ satisfy the following properties. There exists a bound $A > 0$ and constants $M, \eta > 0$ such that for $Q \in \NN$ square-free with $\gcd(Q,\prod_{p \leq  A}p) = 1$ and for $\Upsilon \subset \mathcal{X}(\ZZ/Q\ZZ)$, we have 
\begin{equation} 
\label
{eqn:eff_equ}
\frac{\#\{ x \in \mathcal{X}(\ZZ): H(x) \leq B, x \bmod Q \in \Upsilon\}}
{\#\{ x \in \mathcal{X}(\ZZ): H(x) \leq B\}}
= \frac{\#\Upsilon}{\# \mathcal{X}(\ZZ/Q\ZZ)} + O(Q^{M} B^{-\eta})
\end{equation}
as $B \to \infty$. We call this condition \emph{effective equidistribution}, as it says that the solutions are equidistributed in congruence classes with an explicit error term. For our applications it does not matter how large $M$ is, since we will take $Q$ with $\log Q = o(\log B)$. 
This property holds for example for affine space, projective space \cite[Prop.~2.1]{LS19},
Birch range complete intersections (see \S\ref{sec:Birch}) and a general class of symmetric varieties (see \S\ref{sec:symmvar}).

Let $\mathcal{Z} \subset \mathcal{X}$ be a closed subscheme. For $x \in \mathcal{X}(\ZZ)\setminus \mathcal{Z}(\ZZ)$, we define
\begin{equation} \label{def:omega_Z}
\omega_\mathcal{Z}(x) = \#\{ p : x \bmod p \in \mathcal{Z}(\FF_p)\}.
\end{equation}
The condition $x \notin \mathcal{Z}(\ZZ)$ is easily seen to imply that the number of such primes is finite, hence this is well-defined. Note that $\omega_\mathcal{Z}(x) = \omega_{\mathcal{Z}_{\mathrm{red}}}(x)$ where $\mathcal{Z}_{\mathrm{red}}$ denotes the reduced subscheme underlying $\mathcal{Z}$. In particular, we may always assume that $\mathcal{Z}$ is reduced.

Taking $\mathcal{X} = \mathbb{A}^1_{\ZZ}$ and $\mathcal{Z}$ the origin, this recovers the classical number of prime divisors function $\omega$ used in \S\ref{sec:EK_intro}. Taking $\mathcal{Z}$ to be the coordinate hyperplane $x_i =0 $, we obtain the function $\omega(x_i)$ studied in \S\ref{sec:coordinates}. This is an important change of viewpoint, which makes clear that 
$\omega(x_i)$ 
actually 
has an intrinsic geometric definition. A natural question is how the geometry affects the distribution of $\omega_{\mathcal{Z}}$; as we shall soon see, the geometry determines everything and there is a natural geometric interpretation for all the results in \S\ref{sec:coordinates}.

In Proposition \ref{prop:Serre} we study the average order of this function for a flat closed subscheme $\mathcal{Z} \subset \mathcal{X}$. If $\mathcal{Z}$ is \emph{not} a divisor, then $\omega_{\mathcal{Z}}$ has constant average order. The more interesting case is where $\mathcal{Z} = \mathcal{D}$ is a divisor: here $\omega_{\mathcal{D}}$ has average order $c_{\mathcal{D}} \log \log B$, where $c_{\mathcal{D}}$ denotes the number of irreducible components of $\mathcal{D}$. In particular this behaves strikingly like the usual number of primes divisors of an integer. 

Our main theorem on integral points is an analogue of Erd\H{o}s--Kac's result for our function $\omega_{\mathcal{D}}$. However, given that there are many possible choices for $\mathcal{D}$ it is also natural to simultaneously consider finitely many $\mathcal{D}$, and study the correlations between these divisors. The result we obtain shows that there is in fact a multivariate normal distribution, whose covariance matrix is given explicitly in terms of the geometry of the divisors.

\begin{theorem} 
\label
{thm:main}
	Let $X \subset \PP^d_\QQ$ be a quasi-projective variety  with induced height function $H$  and  $\mathcal{X}$ a choice of model for $X$ over $\ZZ$ which satisfy \eqref{eqn:eff_equ}. Let $\Omega_B = \{x \in \mathcal{X}(\ZZ) : H(x) \leq B\}$  
	be
	equipped with the uniform probability measure. 

	Let $D_1,\dots, D_n \subset X$ be a collection of reduced divisors, $\mathcal{D}_i$ their closures in $\mathcal{X}$ and $\mathcal{D}$ the union of the $\mathcal{D}_i$. Let $c_{i,j}$ denote the number of common irreducible components of $D_i$ and $D_j$. Then as $B \to \infty$, the random vectors
	$$ \Omega_B \setminus \mathcal{D}(\ZZ)  \to \RR^n, \quad x \mapsto \left(\frac{\omega_{\mathcal{D}_1}(x) - c_{1,1}\log \log B}{\sqrt{c_{1,1}\log \log B}}, \dots, \frac{\omega_{\mathcal{D}_n}(x) - c_{n,n}\log \log B}{\sqrt{c_{n,n}\log \log B}}\right)$$
	converge in distribution to a central multivariate normal distribution with covariance matrix whose $(i,j)$-entry is $(c_{i,j}/ \sqrt{c_{i,i} c_{j,j}})$.
	
	Moreover, let $\mathcal{R} = \langle D_1,\dots, D_n\rangle \subset \Div X$ be the group of divisors of $X$ generated by the $D_i$ and let $r$ be the rank of $\mathcal{R}$. Then the covariance matrix has rank $r$.
\end{theorem}

Theorem \ref{thm:main} gives a much more general setting than the results mentioned earlier in the introduction; it allows one to also obtain results where the $x_i$ are replaced by arbitrary polynomials. For example, we obtain the following immediate corollary of Theorem \ref{thm:main}.

\begin{corollary} \label{cor:polynomials}
	Let $f_1, \dots, f_n \in \ZZ[x_1,\dots,x_d]$ and 
	$c_{i,j}$ denote the number of irreducible primitive
	non-constant polynomials $f$ with $f \mid f_i$ and $f \mid f_j$.
	Let	
	$\Omega_B = \{\x \in \ZZ^d : H(\x) \leq B, f_1(\x)\cdots f_n(\x) \neq 0\}$ 
	be
	equipped with the uniform probability measure. 	
	As $B \to \infty$, the random vectors
	$$ \Omega_B \to \RR^n, \quad \x \mapsto \left(\frac{\omega(f_i(\x)) - c_{i,i}\log \log B}{\sqrt{c_{i,i}\log \log B}}\right)_{i=1,\dots,n}$$
	converge in distribution to a central multivariate normal distribution with covariance matrix whose $(i,j)$-entry is $(c_{i,j}/ \sqrt{c_{i,i} c_{j,j}})$.
\end{corollary}

Corollary \ref{cor:polynomials} generalises numerous special cases already known in the literature. The case $n=d=1$ and $f_1$ is irreducible is due to Halberstam \cite[Thm.~3]{Hal56}. The case $n=2,d=1$ and $f_1(x) = x, f_2(x) = x+1$ is also due to Halberstam \cite[Thm.~1]{Hal56}
and
LeVeque \cite[\S4]{Lev49}.
The case $n=1$ and $f_1$ is a product of geometrically irreducible polynomials is due to Xiong \cite[Thm.~1]{Xio09}. The case $d=1$ and the $f_i$ pairwise coprime is due to Tanaka \cite{Tan55}. All these cases either concern a univariate normal distribution, or a multivariate distribution with identity covariance matrix. Our results give a unified proof of all these special cases, and apply in much greater generality.

\begin{remark}
	The covariance matrix in Theorem \ref{thm:main} equals the identity matrix
	if and only if each pair of distinct divisors $D_i$ and $D_j$ have
	no irreducible component in common.
\end{remark}

\begin{remark}
	Our assumption \eqref{eqn:eff_equ} implies that the map 
	$\mathcal{X}(\ZZ) \to \mathcal{X}(\FF_p)$ is surjective for all but finitely many primes $p$; this may be viewed as a weak form 
	of strong approximation.
	However \eqref{eqn:eff_equ} does not imply strong approximation, since  our condition 
	may fail at finitely many primes and we do need require any information
	modulo higher powers of $p$.
\end{remark}

\begin{remark}
	Our method shows that is is possible to replace $\mathcal{X}(\ZZ)$ in \eqref{eqn:eff_equ} by the assumption that there exists some subset $\Omega \subset \mathcal{X}(\ZZ)$ which satisfies  \eqref{eqn:eff_equ}. 
	In particular, one can also consider cases in which there are  accumulating subvarietes or thin subsets.
\end{remark}

\begin{remark} \label{rem:d-uple}
	Let us emphasise that Theorem \ref{thm:main} makes clear that it is
	really the geometric properties of the chosen divisors, 
	rather than the geometry of the
	underlying variety, which determines the covariance matrix.
	For example, let $X \subset \PP^n$ be as in Theorem \ref{thm:CI},
	with coordinates $x_i$.
	We apply the $d$-uple embedding $X \subset \PP^n \subset \PP^{N}$
	for some $d> 1$, where $N = \binom {n+d}d - 1$ and we take the coordinates
	$y_i$ on $\PP^{N}$. Then applying Theorem \ref{thm:main} to $X$
	with respect to coordinate hyperplanes $y_i = 0$, we obtain a covariance
	matrix which is no longer diagonal; indeed, this is exactly the same
	as applying Theorem \ref{thm:main} to the divisors 
	$x_0^{d_0} \cdots x_n^{d_n} = 0$, running over all monomials of degree
	$d$, whence it is easily seen that the covariance matrix is no
	longer diagonal.
\end{remark}

\subsection{Outline of the paper}

In \S \ref{sec:EK} we state our most general theorem (Theorem \ref{thm:mainthmr}), which is a multivariate version of the Erd\H{o}s--Kac theorem for integer sequences satisfying certain hypotheses, and may be viewed as a multidimensional version of Billingsley's work~\cite[\S 3]{MR0466055}. The statement  is very involved, in order to allow for the greatest flexibility for applications. To help the reader, we therefore state a simplified version first in Theorem \ref{thm:mainthmrqw}. This section is dedicated to the proofs of Theorems \ref{thm:mainthmrqw} and \ref{thm:mainthmr}, and is the technical heart of the paper.

In \S \ref{sec:application} we prove Theorem \ref{thm:main} using Theorem \ref{thm:mainthmrqw}. The final \S \ref{sec:examples} concerns various example applications of Theorem \ref{thm:main} to proving the remaining results stated in the introduction. We finish with an example of a cubic surface to which our method does not apply, but for which 
we expect an analogue of our results to hold.

\subsection{Notation and conventions} \label{sec:conventions}

\subsubsection*{Number theory}
We say that a function $g: \NN^n \to \CC$ is \emph{multiplicative} 
 if
for all $\b a , \b b \in \mathbb N^n$ we have
\beq
{eq:multip000}{
g(a_1 b_1, \ldots , a_n b_n)
=
g(\b a)g(\b b), \quad
\mbox{if } \gcd(a_1 a_2 \cdots a_n, b_1 b_2 \cdots b_n)=1
.}
For a prime $p$, we denote by $\nu_p$ the $p$-adic valuation.  

\subsubsection*{Algebraic geometry}

Let $X$ be a variety over $\Q$. A \emph{model} of $X$ over $\Z$ is a finite
type scheme $\mathcal{X} \to \Spec \Z$ together with a choice of isomorphism $X \cong \mathcal{X}_{\Q}$.

\subsubsection*{Probability theory}

\begin{definition} 
\label{def:multinormal}
A random vector $(X_1,\dots,X_n): \Omega \to \R^n$ has a \emph{multivariate normal distribution}
if for every $\b t \in \R^n$
the random variable $\sum_{i=1}^n
t_i X_i$ has a univariate normal distribution.
\end{definition}

Note that for some $\b t$ the random variable  $\sum_{i=1}^n t_i X_i$ may follow a Dirac delta distribution; by convention one views this as a univariate normal distribution with variance $0$. In this case, the associated probability measure will be supported on some affine subspace of $\RR^n$.

A multivariate normal distribution is uniquely determined by its mean vector  $\boldsymbol{\mu}$ and its covariance matrix $\boldsymbol\Sigma$, whose $(i,j)$-entry is $\Cov[X_i , X_j]$. We denote by $\c N(\boldsymbol{\mu}, \boldsymbol\Sigma)$ the associated probability measure on $\RR^n$. A \emph{central} multivariate normal distribution is one with zero mean vector. A \emph{standard} multivariate normal distribution is one with zero mean vector and covariance matrix given by the identity matrix.

We use the notation $\Rightarrow$ to denote convergence in distribution of a sequence of random variables, i.e.~if the corresponding sequence of probability measures convergences weakly.

\subsection{Acknowledgements}
We thank
Carlo Pagano and Zeev Rudnick for helpful comments and suggestions.
The first-named author is supported by the Austrian Science Fund (FWF), projects F-5512 and Y-901.
The second-named author is supported by EPSRC grant EP/R021422/2.
The first and third-named authors acknowledge the 
support of the Max Planck Institute for Mathematics,
where a large part
of this work was carried out. We wish to thank the anonymous referee for helpful remarks that helped to improve various parts of the paper.
 
\section{A multivariate Erd\H{o}s--Kac theorem} \label{sec:EK}
In this section we provide 
a multidimensional generalisation of the Erd\H os--Kac theorem  
for general integer sequences. 
Our main result (Theorem~\ref{thm:mainthmr}),
     proves 
 that   multiple 
additive functions 
  evaluated 
   at 
integer 
sequences
 defined on an arbitrary set
and
  well-distributed in      
arithmetic progressions of very 
small moduli obey
  a multivariate normal distribution. We first give 
a simplified version (Theorem \ref{thm:mainthmrqw}) which is sufficient for many applications, to help ease the reader into the more general technical statement. 

\subsection{Simplified version of the main theorem}
Let $\Omega$ be an infinite set
and assume that  we are given a function 
$h:\Omega\to \R_{\geq 0}$
with 
 \beq
{def:northcot2qwe}
{
N(B) \text{ finite for all } B\geq 0, 
\quad \mbox{where } 
N(B):=\#\{a\in \Omega: h(a)\leq B\}
.}
Note that as $\Omega$ is infinite we have $N(B) \to \infty$. Moreover \eqref{def:northcot2qwe} implies that $\Omega$ is countable.
For each $B\geq 0$ we equip the set $\Omega$ with the structure of a probability space using the discrete $\sigma$-algebra and probability measure 
\[\b P_B\l[S\r]:=\frac
{\#\{a\in S    : h(a) \leq B\} }
{\#\{a\in \Omega : h(a) \leq B\}}, 
\quad S \subseteq \Omega
.\] 
Note that this measure is supported on the finite set $\#\{a\in \Omega : h(a) \leq B\}$, where it induces the uniform measure.
Next, we assume that 
we are given
$n\in \N$ and a function
\beq
{def:omomqwe}
{
m:\Omega
\to \N^n,
a \in \Omega\mapsto
(m_1(a),\ldots,m_n(a))
.}
We are interested in studying the distribution of the vector 
\[
\l(
\omega(m_1(a) ),
\ldots,
\omega(m_n(a) )
\r)
.\]
As with the classical Erd\H os--Kac theorem,
we need to normalise
by suitable factors first. We
have to assume some kind of 
regularity among the values of $m_i(a)$, namely, 
that there exists $A\in \R$
such that 
the following limit exists for all $\b d \in \N^n$
satisfying
$p\mid d_1\cdots d_n \Rightarrow p>A$, 
\beq
{def:gdensityqwe}
{ \lim_{B\to+\infty}
\frac{\#\l\{
a\in \Omega: h(a) \leq B,
d_i \mid m_i (a)  \, \forall  \,1 \leq i \leq n
\r\}}{
N(B)}
 =:g(\b d).}
The reason for  assuming~\eqref{def:gdensityqwe}
only
for moduli without small prime factors is that in certain situations it is convenient to ignore 
small `bad' primes.
We furthermore assume that 
 \beq
{eq:::defmultqwe}
{  g 
\text{ is multiplicative in the sense of }
\eqref{eq:multip000} }
and   
  extend $g$ to $\N^n$
by setting it equal to $0$ for $\b d $ 
such that $d_1 \cdots d_n$ has a prime factor $p\leq A$.
For any $1\leq i, j  \leq n$ 
we let 
\[
g_i(d):=g(1,\dots,1,\underset{\underset{i}{\uparrow}}{d},1,\dots,1)
\text{ and } 
g_{i,j}(d):=g(1, \dots, 1, \underset{\underset{i}{\uparrow}}{d},1,\dots,1, \underset{\underset{j}{\uparrow}}{d},1,\dots,1 )
.\] 
We now assume that 
for every $1\leq i \leq n $
we have 
 \beq
{eq:asumptnqwe}
{
\sum_{p >T } 
g_i(p)^2=
O\l(\frac{1}{\log T}\r)
\text{ and } 
 \sum_{p\leq T}
g_i(p)= 
c_i \log \log T +c'_i
+
O\l(\frac{1}{\log T}\r), }
for some $c_i  >0, c'_i \in \R $.  
This assumption is highly typical and usually
met in   
 sieve theory problems
as it corresponds to a sieve of `dimension' $c_i$.
In light of~\eqref{def:gdensityqwe}
the sum $\sum_{p\leq m_i(a)} g_i(p) $
should be thought of as approximating the
expected value
of $ \omega(m_i(a))$ as one samples over suitably 
many $a\in \Omega$.

The main arithmetic input in our theorem 
is   a statement regarding
the speed of
convergence in~\eqref{def:gdensityqwe}. Namely,
we 
 define $\c R
(\b d, B)$ for each $\b{d} \in \N^n$ 
and $B\geq 1$
via \begin{equation} 
\label
{def:levelofdistribvcfdgtqwe}
\c R(\b{d},B)
:=
\#\l\{a\in \Omega: h(a) \leq B, 
 d_i \mid m_i(a) \, \forall  1\leq i \leq n
 \r\}-
g(\b d)
N(B).
\end{equation} 
We demand that  $\c R(\b d , B)$ is asymptotically
smaller than $N(B)$ for most $\b d $ that are smaller than the 
`typical size' of the $m_i(a)$. To make 
this notion
precise, we first    
  call $\c F(B)$ the typical size of $
\max_{1\leq i \leq n }  m_i(a)
$, namely we assume 
there
exists a function 
$\c F:\R_{\geq 1}\to \R$
 with 
 \beq
{def:northcot2hgf8sqwe}
{
\hspace{-0,3cm}
 \lim_{B\to 
\infty}
\frac{1}{N(B)}
\#\l\{
a\in \Omega:
h(a) \leq B,    
 \max_{1\leq i \leq n } m_i(a) 
\leq
 \c F(B) 
\r\}=
1
.} It will turn out that the other assumptions in our set-up
ensure that 
$\lim_{B\to\infty} \c F(B)=+\infty$. Secondly, 
we assume that 
the sequences $m_i(a)$
are well-distributed in arithmetic progressions
whose modulus is
small 
compared to $\c F(B)$.
Namely, let 
\beq
 {def:northcot2hgf8sqwerwe5we}
{
\epsilon(B):=
\frac{
 \log \log \log  \c F(B) 
}{\sqrt{ \log \log  \c F(B)  } }
} and 
assume that for the same 
$A\in \R$
as above,
the following estimate is valid  
for all $\gamma>0$
 \beq
{eq:thebigassumptnqwe}
{ 
  \sum_{\substack{
\b d \in \N^n
\\
|\b d |\leq   \c F(B) ^{\epsilon(B)}
\\
p\mid d_1 \cdots d_n 
\Rightarrow p>A
}}
\mu(d_1)^2
\cdots
\mu(d_n)^2 
|\c R(\b d,B)|
 \ll_\gamma 
\frac{ N(B) }{(\log \log \c F(B))^\gamma}
} 
with an implied constant that is independent of $B$.
Assumption~\eqref{eq:thebigassumptnqwe}
is the     main arithmetic 
input  needed in our main results
(see Remark~\ref{rem:zerolevelofdistr}).
Define the function 
$
\b K
:\Omega\to \R^n$ via 
\beq
{defeqkqwe}
{
\b K(a):=
\l(
\frac{  \omega(m_1(a) )   - c_1 \log \log \c F(B)  }{\sqrt{c_1 \log \log \c F(B)}}
,\ldots,
\frac{ \omega(m_n(a) )   -  c_n \log \log \c F(B)  }{\sqrt{c_n \log \log \c F(B)}}
\r)
.}
 This is the promised normalisation. Our result is as follows. 
\begin{theorem} 
\label
{thm:mainthmrqw}
Let $n\in \N$ 
and 
assume that we are given 
 a set 
$\Omega$,
a real number $A$
and 
functions 
$ h,m, g, \c F$ 
such that~\eqref{def:northcot2qwe},
~\eqref{def:omomqwe},
~\eqref{def:gdensityqwe},
~\eqref{eq:::defmultqwe},
~\eqref{eq:asumptnqwe},
~\eqref{def:northcot2hgf8sqwe}
and~\eqref{eq:thebigassumptnqwe} 
hold.
Furthermore,
assume that for every 
$1\leq i, j \leq n $
the following limit exists,
\begin{equation} \label{eqn:gij}
\lim_{T\to+\infty}
\frac{
\sum_{p\leq T
}
 g_{i,j}(p) 
  }{
(\sum_{p\leq T} g_i(p))^{1/2}
(\sum_{p\leq T} g_j(p))^{1/2}
   }.
\end{equation}Then the random vectors
\beq
{eq:thelimitqwe}
{
(\Omega, \mathbf{P}_B) \to \RR^n, \quad a \mapsto \b K (a),
}
converge in distribution as $B \to \infty$ to a central multivariate normal distribution with covariance matrix 
$\boldsymbol{\Sigma}$
whose $(i,j)$-entry is the limit~
\eqref
{eqn:gij}.
 \end{theorem}

There are three 
noteworthy 
aspects 
in 
Theorem~\ref{thm:mainthmrqw}.
Firstly, the simplest case with 
$n=1$
 applies to functions defined on a \textit{general}
set $\Omega$, hence it 
recovers normal
distribution results
  related to 
irreducible polynomials \cite[Thm.~3]{Hal56},
values of irreducible polynomials at primes \cite{MR73627}
and 
  entries of matrices \cite{delbz18}. It also applies to new situations, such as the coordinates of integer zeros of affine algebraic varieties 
 that do not necessarily have a group structure.

Secondly,
Theorem~\ref{thm:mainthmrqw}
studies
\textit{multidimensional}
normal laws for arithmetic functions. The only related example that we could find 
in the literature is due to Halberstam~\cite[Thm.~1]{Hal56} and 
LeVeque \cite[\S4]{Lev49}
regarding $(\omega(m), \omega(m+1) )$
and its  generalisation
given by 
Tanaka~\cite{Tan55}
regarding 
$(\omega(f_1(m) ),\ldots, \omega(f_n(m)))$
for non-constant integer irreducible 
 polynomials $f_i$ that are relatively coprime. These 
results 
are 
 recovered by 
our theorem by taking $\Omega= \N , m_i(a)=f_i(a)$ for $1\leq i \leq n $,
and the covariance matrix is the
$n\times n $ 
identity matrix.

Thirdly, the covariance matrix is the identity if and
only if 
the  sequences
$\omega(m_i(a))$ and $\omega(m_j (a)) $  are `uncorrelated' for all $i\neq j $.
  Such a phenomenon is however not present 
in many situations (such as the prime factors of coordinates of affine algebraic varieties)
and one must therefore obtain a general Erd\H os--Kac law that would apply to situations with non-vanishing
correlations. This is the most important new aspect of  
Theorem~\ref{thm:mainthmrqw}, namely, that it covers 
multivariate normal distributions with 
\textit{arbitrary}  
covariance matrix.

\begin{remark}
\label
{rem:zerolevelofdistr}
Assumption~\eqref{eq:thebigassumptnqwe}
resembles  
  a 
  \textit{level of distribution} condition
in sieve theory.
In typical 
situations one takes
 $\c F(B) = N(B)^c    $ for some fixed $c>0$, where the  
size
 condition  on $\b d $
becomes 
$
|\b d | 
\leq    
\c F(B) ^{\epsilon(B)} 
=
  N(B)^{o(1)} 
$.
This is much lighter than the usually stricter assumption in classical
sieve theory problems, where
a positive exponent of level of distribution is required, i.e.~one requires the same error term but with the summation 
over $\b d $ with $|\b d| \leq N(B)^\alpha$ for some fixed $\alpha>0$.
Note that if there exist   $\eta>0$ and $M>0$  such that 
\[
\frac{\#\{
a\in \Omega :  h(a) \leq B,d_i \mid m_i(a) \, \forall 1\leq i \leq n
\}}{
\#\{
a\in \Omega: h(a) \leq B
\}
}
=g(\b d)
+O\l(
N(B)^{-\eta}
(\max_{1\leq i \leq n } d_i)^M
\r)
 \]
and if 
$\c F(B) = N(B)^c$, then~\eqref{eq:thebigassumptnqwe}
holds
due  to the estimate 
\[
\sum_{|\b d | \leq \c F(B)^{\epsilon(B)} }
|\c R(\b d ,B)|
\ll
N(B)^{1-\eta}
\sum_{|\b d | \leq \c F(B)^{\epsilon(B)} }
(\max_{1\leq i \leq n } d_i)^M
\ll  N(B)^{1-\eta}
\c F(B)^{\epsilon(B) (M+n) } .\] 
Since $\epsilon(B)=o(1)$, this is  $ \ll N(B)^{1-\eta/2}$, which, for every $\gamma>0 $ is 
$$ o\l(\frac{N(B)}{(\log \log N(B))^\gamma} \r )=O\l(\frac{N(B)}{(\log \log \c F(B))^\gamma} \r)   $$ owing to the equality $\c F(B)=N(B)^c$.
 \end{remark}

\subsection{The main theorem}
\label
{subs:mainthrem}  
We now state 
the main technical result in the present paper; it is a
general version 
of 
Theorem~\ref{thm:mainthmrqw}. 
Let $\Omega$ be an infinite set
and assume that for every $B\in \R_{\geq 1}$
we are given a function 
$\chi_B:\Omega\to \R_{\geq 0}$ 
such that 
\beq
{def:northcot2}
{B\geq 1 
\Rightarrow 
 \{a\in \Omega: \chi_B(a)>0 \} \text{ finite}
.}
In applications the function $\chi_B(x)$ will either denote  the characteristic function of  
elements $x$ having `height' bounded by $B$
or it will be a smooth `weight' function of the form $w(x/B)$.
We also demand that 
\beq
{def:northcot3}
{\lim_{B\to+\infty}
\sum_{a\in \Omega} \chi_B(a)
=+\infty
.}
For each $B\geq 0$ we equip the set $\Omega$ with the structure of a probability space using the discrete $\sigma$-algebra and probability measure 
\[\b P_B\l[S\r]:=\frac{\sum_{a\in S} \chi_B(a)}{\sum_{a\in \Omega} \chi_B(a)},
\quad S \subseteq \Omega
.\]
Assume that 
 $M:\R_{\geq 0}\to \R_{\geq 0}$
is any  function
satisfying 
\beq
{def:mmde}
{ 
\lim_{B\to+\infty}
\frac{
\sum_{\substack{  a\in \Omega
 }}\chi_B(a)  }
{M(B)}=1
.}
Next, we assume that 
we are given
$n\in \N$ and a function
\beq
{def:omom}
{
m:\Omega
\to \N^n,
a \in \Omega\mapsto
(m_1(a),\ldots,m_n(a))
.}
We will find  general  assumptions which 
 ensure that certain functions  
display Gaussian behaviour simultaneously for all $i$
when evaluated at $m_i(a)$.  
 We first need
the following function $g$, that contains information on the 
divisors of typical values of $m_i(a)$.   
\begin{definition}[The density function $g$]
\label
{def:g_i}
We 
assume that there exists $A\in \R$
such that 
the following limit exists for all $\b d \in \N^n$
with $p\mid d_1\cdots d_n \Rightarrow p>A$, 
\beq
{def:gdensity}
{
   \lim_{B\to+\infty}
\frac{1}{
 M(B)
}
\sum_{\substack{  a\in \Omega\\   d_i \mid m_i(a) \, \forall 1\leq i \leq n }} \chi_B(a).} 
We define $g:\{\b d \in \N^n: p\mid d_1 \cdots d_n \Rightarrow p>A\}
 \to \R$ as the value of this limit.
We extend $g$ to $\N^n$
by setting it equal to $0$ for $\b d $ 
such that $d_1 \cdots d_n$ has a prime factor $p\leq A$.
Furthermore,
we
assume that 
\beq
{eq:::defmult}
{  g 
\text{ is multiplicative in the sense of }
\eqref{eq:multip000}. } 
\end{definition} 
Let us introduce
the arithmetic functions 
whose values at $m_i(a)$ 
we shall study.
 These functions will be of the form $\sum_{p\mid m_i(a) } 
\theta_i(p)$, where the sum is taken over 
prime divisors $p$
and 
 $\theta_i (p)$ are bounded functions. 
These function clearly generalise $\omega$ as can be seen by taking
$A=0$ and
 $\theta_i(p)=1$ for all $p$.  To be precise, 
we 
assume that we are given
functions
$\theta_1,\ldots,\theta_n  
$ 
   defined on the primes,
taking values on $\R$  
and that 
there exists 
$\Theta\in \R$ with 
\beq
{assumption:thet}
{|\theta_i(p)| \leq
 \Theta
\text{ for all }
1\leq i \leq n 
\text{ and primes } p.} 
For any $S\subset \{1,\ldots, n\}$ 
and $b\in \N$
we define $g_S:\N\to \R$
via   
 \beq
{def:brandemb3}
{g_S(b):=g\l(
1 + (b-1)\mathds 1_{S}(1),\ldots, 
1 + (b-1)\mathds 1_{S}(i),\ldots, 
1 + (b-1)\mathds 1_{S}(n)
\r)
,}
i.e.~we put $b$ in position $i$ if $i \in S$ and we put $1$ otherwise.
We furthermore define  
\beq
{def:limitations}
{
\c M_i(T):=
 \sum_{p\leq T} \theta_i(p) g_i(p), \ \ (T\geq 0,  1\leq i \leq n ) .}
The function  $\c M_i(T)$ 
approximates the `mean'
of $ \sum_{p\mid m_i(a) } \theta_i(p)$ as one samples over suitably 
many $a\in \Omega$.
 In addition to
these means we shall also need to 
consider the analogous of 
`variances' 
$\c V_i(T)^2$,  
thus we let     
\beq
{eq:Trilogy Suite Op. 5 - Yngwie Malmsteen}
{
\c V_i(T):=
\Big( 
\sum_{p\leq  T}   \theta_i(p)^2  g_i(p)   \l(1-g_i(p)    \r)
\Big)
^{1/2}, \ \ (T\geq 0,  1\leq i \leq n ).
}
 We assume that for all $i=1,\ldots,n$ we have
 \beq
{eq:asumptn}
{
\lim_{T\to+\infty}
\c V_i(T)
 =+\infty
.}
Let us define the function 
$
\b K
:\Omega\to \R^n$ via 
\beq
{defeqk}
{
\b K(a):=
\l(
\frac{\l(\sum_{p\mid m_1(a)} \theta_1(p)  \r)-\c M_1(m_1(a))}{\c V_1(m_1(a))}
,\ldots,
\frac{\l(\sum_{p\mid m_n(a)} \theta_n(p)  \r)-\c M_n(m_n(a))}{\c V_n(m_n(a))}
\r)
.}
If $\mathcal{V}_i(m_i(a)) = 0$, then by convention we take the $i$th entry to be $1$ (note that our later assumptions will imply that for any $i$ the event $\mathcal{V}_i(m_i(a)) = 0$ has probability $0$) 

We will study
the 
behaviour of the functions $\sum_{p\mid m_i(a) } \theta_i(p)$ 
simultaneously for all $i$ and 
as $a$ ranges over $\Omega$. To
 make the notation easier in what follows we 
 normalised
 these functions 
by first 
centering around their `expected mean'  $\c M_i$ and then dividing by the `standard deviation' $\c V_i$.
We 
 define $\c R
(\b d, B)$ for each $\b{d} \in \N^n$ 
and $B\geq 0$
via \begin{equation} 
\label
{def:levelofdistribvcfdgt}
\c R(\b{d},B)
:=\l(
\sum_{\substack{  a\in \Omega\\  
d_i \mid m_i(a) \, \forall 1\leq i \leq n
 }}
\chi_B(a)
\r)
-
g(\b d)
M(B).
\end{equation}
Our result 
will hold if the size of 
$\c R(\b d ,B)$ 
is relatively
 small compared to 
$\c M_i(B)$ and $\c V_i(B)$
as one 
averages over small $\b d $. To make this
precise 
we need the following piece of 
notation. 
\begin{definition}
[Truncation pairs]
\label
{def:tranc}
We say that a pair of functions $(\c F,{\psi})$ with
$\c F:\R_{\geq 0}\to \R$ and
   $\psi:\R\to (0,1]$ is a \emph{truncation pair} if the following is satisfied.
First
\beq
{def:northcot2hgf8s}
{
 \lim_{B\to+\infty}
\frac{1}{\sum_{a\in \Omega} \chi_B(a)}
\sum_{\substack{a\in \Omega\\ m_i(a) \leq 
\c F(B)\, \forall i}}\chi_B(a)=
1
.}
Next
\beq
{eq:Flots du Danube}
{
\lim_{B\to+\infty}
{\c F(B)}^{\psi(B)} =+\infty
,}
\beq
{eq:Albioni - Oboe Concerto in D minor, op.9 no.2
1.  I  Allegro e non presto}
{
\lim_{B\to+\infty}
\frac{1  }{\psi(B)
\c V_i(\c F(B)^{\psi(B)})
}
=0
, \quad
\frac
{\c M_i(m_i(a))-\c M_i(\c F(B)^{\psi(B)}  )  }
{
\c V_i(\c F(B)^{\psi(B)})
}
\Rightarrow 0}
and 
\beq
{eq:assumptionalbinioni}
{
\frac
{\c V_i(m_i(a))}
{
\c V_i(\c F(B)^{\psi(B)})
}
\Rightarrow
1.}
Lastly, we 
assume  that 
for every $k_1,\ldots,k_n \in \N$ we have
\beq
{eq:thebigassumptn}
{
 \lim_{B\to+\infty}
 \frac{\prod_{i=1}^n
\l(
 \c M_i(\c F(B)^{\psi(B)})+\Theta ) 
 \r)^{k_i}}{M(B)\prod_{i=1}^n
\c V_i(\c F(B)^{\psi(B)})^{k_i}
}
 \hspace{-0,1 cm}
  \sum_{\substack{
\b d \in \N^n
\\
\!
\!
\eqref{eq:manzacafe}
}}
 \l(
g(\b d)
|\c R((1,\ldots,1),B)| 
+
|\c R(\b d,B)|
\r)
 =0,} 
where the summation is over $\b d \in \N^n$ with 
\beq
{eq:manzacafe}
{
\begin{cases}  
p\mid d_1\cdots d_n \Rightarrow A<p    \leq {\c F(B)}^{\psi(B)},   \\ 
\mu(d_i)^2=1, \quad \forall 1\leq i \leq n,  \\ 
\omega(d_i)\leq k_i, \quad \forall 1\leq i \leq n.
\end{cases}}

\end{definition}

 \begin{theorem} 
\label
{thm:mainthmr}
Assume that we are given 
$n \in \N$,
an infinite set 
$\Omega$,
 a function 
  $M:\R_{\geq 1}\to \R_{\geq 0}$ 
and for all $B\in \R_{\geq 1}$ a function 
$\chi_B:\Omega\to \R_{\geq 0}$ 
such that 
for any $B\in \R_{\geq 1}$
the assumptions
~\eqref{def:northcot2},
~\eqref{def:northcot3}
and~\eqref{def:mmde}
are satisfied. Assume further that 
we are given
a function 
$m:\Omega \to \N^n$
satisfying~\eqref{def:gdensity},
a real number $A$ and a 
 map 
$g:\N^n\to \R$
satisfying~\eqref{eq:::defmult} 
and functions $\theta_1,\ldots,\theta_n$
defined on the primes that take values in $\R$
that fulfil~\eqref{assumption:thet} and \eqref{eq:asumptn}.  
Assume that there exists a truncation pair
  $(\c F,\psi)$ satisfying~\eqref{def:northcot2hgf8s}-\eqref{eq:thebigassumptn}
and that 
for every $1\leq i, j \leq  n $ 
the following limit exists,
\beq
{A. VIVALDI:Filiae maestae Jerusalem RV 638}
{ 
\lim_{T\to+\infty}
\frac{
\sum_{p\leq T
}
\theta_i(p) \theta_j(p)
\l(g_{\{i,j\}}(p)- g_i(p) g_j(p)  \r)
  }{
\c V_i(T)
\c V_j(T)
   }
.}
Then the random vectors
\beq
{eq:thelimit}
{
(\Omega, \mathbf{P}_B) \to \RR^n, \quad a \mapsto \b K (a),
}
converge in distribution as $B \to \infty$ to a central multivariate normal distribution with covariance matrix 
$\boldsymbol{\Sigma}$
whose $(i,j)$-entry is the limit~
\eqref
{A. VIVALDI:Filiae maestae Jerusalem RV 638}.
 \end{theorem}

\subsection{The proof of  Theorem~\ref{thm:mainthmr}} 
\label
{s:prfngen} 

To prove the result, we shall use the \emph{method of moments}. Specifically, the normal distribution has the special property that it is completely determined by its moments. Therefore it suffices to calculate the moments in our case. Our precise application  is slightly more delicate, and we instead approximate with a sum of random variables, and use a version of the method of moments due to Billingsey (Lemma \ref{lem:billbill}).

Our strategy consists of showing that 
for all $\b t \in \R^n$
the random variable on $\Omega$ given by 
\beq
{eq:kaloe?}
{
\sum_{i=1}^n t_i 
\Bigg(
\frac{\l(\sum_{p\mid m_i(a)} \theta_i(p)  \r)-\c M_i(m_i(a))}{\c V_i(m_i(a))}
\Bigg)
 }
converges in distribution as $B \to \infty$ to a suitable
linear combination of univariate normal distributions. 
To be able to     
use the level-of-distribution
property~\eqref{eq:thebigassumptn},
we show that we can
restrict the
size
of the primes $p\mid m_i(a)$ to the range $p\leq \c F(B)^{\psi(B)}$. 
\begin{lemma}
\label
{lem:albinionioboe7}
For all $\b t\in \R^n$ we have 
\[
\sum_{i=1}^n t_i 
\frac{\l(\sum_{p\mid m_i(a)} \theta_i(p)  \r)-\c M_i(m_i(a))}{\c V_i(m_i(a))}
  -
\sum_{i=1}^n t_i 
\frac{\l(\sum_{
   A<  
p\mid m_i(a), p\leq 
\c F(B)^{\psi(B)}
} \theta_i(p)  \r)-\c M_i(\c F(B)^{\psi(B)})}{\c V_i(\c F(B)^{\psi(B)})  }
\Rightarrow 0
.\]
\end{lemma}
\begin{proof}
By Slutsky's theorem \cite[3.2.13]{Dur19}, if $X_m,Y_m$ are sequences of random variables with $X_m\Rightarrow 0$ and $Y_m\Rightarrow 0$
then  $X_m+ Y_m\Rightarrow 0$. Therefore, it suffices to prove that 
\[ 
 \frac{\l(\sum_{p\mid m_i(a)} \theta_i(p)  \r)-\c M_i(m_i(a))}{\c V_i(m_i(a))}
  -
\frac{\l(\sum_{p\mid m_i(a), 
 {  A<  }
p\leq 
\c F(B)^{\psi(B)}
} \theta_i(p)  \r)-\c M_i(\c F(B)^{\psi(B)})}{\c V_i(\c F(B)^{\psi(B)})  }
\Rightarrow 0
\] 
for every $i$.
Using~\eqref{eq:assumptionalbinioni}
we see that it is sufficient to prove that 
\begin{equation} \label{eqn:fuck}
 \hspace{-0,4cm}
\frac{
\l(\sum_{p\mid m_i(a)} \theta_i(p)  \r)
-\l(\sum_{p\mid m_i(a), 
 {  A<  }
p\leq   \c F(B)^{\psi(B)}   } \theta_i(p)  \r)
+\l(\c M_i(\c F(B)^{\psi(B)})
-
\c M_i(m_i(a))
\r)}
{\c V_i(\c F(B)^{\psi(B)})  } 
\Rightarrow 0
.
\end{equation} 
To prove this we
 use~\eqref{assumption:thet} 
to see that 
$\sum_{p\leq A} \theta_i(p) \ll 
A
\Theta $  and 
  \begin
{align*}
\l(\sum_{p\mid m_i(a)} \theta_i(p)  \r)
-\l(\sum_{p\mid m_i(a), p\leq   \c F(B)^{\psi(B)}   } \theta_i(p)  \r) \ll
&
\Theta
\#\l\{p> \c F(B)^{\psi(B)}   : p\mid m_i(a)\r\}
.\end{align*}
Let $\Omega_0 = \{ a \in \Omega : m_i(a) \leq \mathcal{F}(B)\}$;
note that $\lim_{B\to+\infty} \b P_B(\Omega_0)=1$  by \eqref{def:northcot2hgf8s}. Thus we can use the bound 
$\#\{p>z: p \mid m \} \leq (\log m)/(\log z)$ to see that the numerator in \eqref{eqn:fuck} is 
\begin{align*}
&\ll_{
{ A, }
\Theta}
{ 1+}
\frac{\log m_i(a) }{\log  (\c F(B)^{\psi(B)}   )  }+
\l(\c M_i(m_i(a))-\c M_i( \c F(B)^{\psi(B)}     ) \r)
 \\
& \ll 
\frac{1}{   \psi(B)} 
+
\l(\c M_i(m_i(a))-\c M_i( \c F(B)^{\psi(B)}     ) \r)
.\end{align*} 
The proof is concluded by using~\eqref
{eq:Albioni - Oboe Concerto in D minor, op.9 no.2
1.  I  Allegro e non presto}. 
\end{proof}
For a function $h:\Omega\to \mathbb C$ we define $\mathbb E_B$ as follows,
\[
\mathbb E_B[h]:= 
\frac{1}{\sum_{a\in \Omega} \chi_B(a)}
\sum_{\substack{a\in \Omega }}
\chi_B(a) h(a)
,\]
i.e.~the expected value of $h$ with respect to $\b P_B$.
We begin by reducing the evaluation of moments  
to averages over $\b d $
of
the error
term functions
 $\c R(\b d ,B)$
introduced in~\eqref{def:levelofdistribvcfdgt}. 
\begin{lemma}
\label
{lem:31billmoza}  
For all $B\geq 1$
and 
    $k_1,\ldots,k_n
\in \N$
the following estimate holds with an absolute implied constant,
\begin{align*}
&
\mathbb{E}_B\l[
\prod_{i=1}^n
 \l(\sum_{\substack{
{ A< }
p\leq {\c F(B)}^{\psi(B)}  
\\
p\mid m_i(a)
}}
 \theta_i(p) \r)^{k_i}\r]
-
\sum_{\substack{
1 \leq i \leq n \\ 1 \leq j \leq k_i \\
A< p_{i,j} \leq \mathcal{F}(B)^{\psi(B)} 
}}
g\l(P_1,\ldots, P_n\r)
\prod_{\substack{
1\leq u \leq n \\
1\leq v \leq k_u  
}}
\theta_u(p_{u,v}
)
\\
&\ll
\frac{
 \Theta^{k_1+\cdots+k_n}
}
{M(B)}  
\sum_{\substack{
\b d \in \N^n
\\
\!
\!
\eqref{eq:manzacafe}
}}
 \l(
g(\b d)
|\c R((1,\ldots,1),B)| 
+
|\c R(\b d,B)|
\r)
,
\end{align*}
 where $P_u$ is the radical of 
$\prod_{1\leq v \leq k_u}  
p_{u,v} $
and 
the sum over $p_{i,j}$
is   over
primes. 
\end{lemma}
\begin{proof}
Expanding the $k_i$-th powers gives 
\[
\l(\sum_{a\in \Omega}
\chi_B(a) \r)
\mathbb{E}_B\l[\prod_{i=1}^n
\l(\sum_{\substack{
{ A <  }
p\leq {\c F(B)}^{\psi(B)}  
\\
p\mid m_i(a)
}}\theta_i(p) \r)^{k_i}\r]
=
\hspace{-10pt}
\sum_{\substack{
1 \leq i \leq n \\ 1 \leq j \leq k_i \\
A< p_{i,j} \leq \mathcal{F}(B)^{\psi(B)} 
}}
\l\{\prod_{\substack{
1\leq u \leq n \\
1\leq v \leq k_u  
}}
\theta_u(p_{u,v}
)\r\}
\sum_{\substack{ a\in \Omega \\\forall i: P_i \mid m_i(a)}} \chi_B(a)
.\]
By~\eqref
{def:levelofdistribvcfdgt}
this equals 
\[E+
M(B)\sum_{\substack{
1 \leq i \leq n \\ 1 \leq j \leq k_i \\
A< p_{i,j} \leq \mathcal{F}(B)^{\psi(B)} 
}}
g( P_1,\ldots,P_n)
\prod_{\substack{
1\leq u \leq n \\
1\leq v \leq k_u  
}}
\theta_u(p_{u,v}
)
,\] where 
 $E$ is given 
 by  
\[E:=
\sum_{\substack{
1 \leq i \leq n \\ 1 \leq j \leq k_i \\
A< p_{i,j} \leq \mathcal{F}(B)^{\psi(B)} 
}}
\c R(  ( P_1,\ldots,P_n)   , B) 
\prod_{\substack{
1\leq u \leq n \\
1\leq v \leq k_u  
}}
\theta_u(p_{u,v}
)
.\]
By~\eqref{assumption:thet}
we infer that
\[
|E| \leq  \Theta^{k_1+\cdots+k_n}  
\sum_{\substack{
1 \leq i \leq n \\ 1 \leq j \leq k_i \\
A< p_{i,j} \leq \mathcal{F}(B)^{\psi(B)} 
}}
|\c R(  ( P_1,\ldots,P_n)   , B)   |
\leq 
 \Theta^{k_1+\cdots+k_n}  
  \sum_{\substack{
\b d \in \N^n
\\
\!
\!
\eqref{eq:manzacafe}
}}
  |\c R(\b d, B)|
.\]
To conclude the proof, it follows from \eqref{def:mmde} that
\[
\frac{M(B)}{\sum_{a\in \Omega}
\chi_B(a)}
=
\frac{\sum_{a\in \Omega}
\chi_B(a)
-\c R( (1,\ldots,1),B)
}{\sum_{a\in \Omega}
\chi_B(a)
}
=1+O\l(
\frac{|\c R((1,\ldots,1),B)|
}{M(B)}\r)
.\] 
We deduce  that 
\begin
{align*}
&  \l| 
\mathbb{E}_B
\Bigg[ 
\prod_{i=1}^n
\Bigg(
\sum_{\substack{
p\leq {\c F(B)}^{\psi(B)}  
\\
p\mid m_i(a)
}} \theta_i(p) 
\Bigg)
^{k_i}
\Bigg]
-\sum_{\substack{
1 \leq i \leq n \\ 1 \leq j \leq k_i \\
A< p_{i,j} \leq \mathcal{F}(B)^{\psi(B)} 
}}
g ( P_1,\ldots,P_n)
\prod_{\substack{
1\leq u \leq n \\
1\leq v \leq k_u  
}}
\theta_u(p_{u,v}
)
\r| 
\\ 
&\ll
\frac
{|\c R((1,\ldots,1),B)| }
{M(B)}
\l(\sum_{\substack{
1 \leq i \leq n \\ 1 \leq j \leq k_i \\
A< p_{i,j} \leq \mathcal{F}(B)^{\psi(B)} 
}}
\hspace{-10pt}
g ( P_1,\ldots,P_n)
\prod_{\substack{
1\leq u \leq n \\
1\leq v \leq k_u %ES k_i
}}
\theta_u(p_{u,v}
)
\r)
+
\frac{\Theta^{k_1+\cdots+k_n} }
{M(B)}
  \sum_{\substack{
\b d \in \N^n
\\
\!
\!
\eqref{eq:manzacafe}
}}
 |\c R( \b d, B)|
.\end{align*}
The last sum over $p_{i,j}$
is at most 
$
\Theta^{k_1+\cdots+k_n}
\sum_{\b d \in \N^n,\eqref{eq:manzacafe} }
g(\b d)$,
thus concluding the proof.
\end{proof} 
We need  
to understand the expression 
$ \sum_{p_{i,j} } \prod_{i,j} \theta_i(p_{i,j})
g ( P_1,\ldots,P_n)$ 
in Lemma~\ref{lem:31billmoza}
before proceeding.
This will be based on
interpreting  the function
 $g_S(p)$ in~\eqref{def:gdensity}
as the `probability'
that
$p$ divides each component of 
the vector $(m_i(a))_{i\in S}$ as  $a$ ranges through $\Omega$.
We do this by introducing some auxiliary random vectors.

\begin{lemma}
For every prime  $p   > A$  
there exists
a random vector $\b X_p=(X_{1,p},\ldots, X_{n,p}),$
such that 
\begin{align}
& 
\text{The random vectors }
\b X_{p}  \text{ are independent for all primes } p.
 \label
{eqdef:1}   \\
&
X_{i,p} \text{ is Bernoulli and  takes     values  in } \{0,1\},
\label{eqdef:3} \\
&
\mathrm{Prob}\l[ \bigcap_{i\in S} \l\{X_{i,p}=1  \r\} \r] = g_S(p)
\mbox{ for all } S\subset \{1,\dots, n\}.
\label
{eqdef:4}  
\end{align} 
\end{lemma} 
\begin{proof} 
 We first show that for a fixed prime $p$, there exists a random vector satisfying \eqref{eqdef:3} and  \eqref{eqdef:4}. To do so, let $S \subset \{1,\dots,n\}$ with complement $S^c$. Then we define
\begin{equation} \label{eqn:prob}
\mathrm{Prob}\l[ \bigcap_{i\in S} \l\{X_{i,p}=1  \r\}
\bigcap_{i\in S^c} \l\{X_{i,p}=0  \r\} \r] 
=  \sum_{S' \subset S^c} (-1)^{|S'|}g_{S \cup S'}(p).
\end{equation}
To see that this gives a well-defined random vector, it suffices to show that each probability \eqref{eqn:prob} is non-negative
(that the sum of  
all probabilities equals $1$ follows from a simple inclusion-exclusion argument and the fact that $g_\emptyset(p) = g(1,\dots,1) = 1$). However, by inclusion-exclusion we have
\begin{align*}
 0 \leq \lim_{B\to+\infty}
\frac{1}{ M(B)}
\sum_{\substack{  a\in \Omega\\ p \mid m_i(a),   \, i \in S \\
	p \nmid m_i(a), \, i \in S^c}}\chi_B(a) 
	& =  \lim_{B\to+\infty}\frac{1}{ M(B)}
	\sum_{S' \subset S^c} (-1)^{|S'|}
	\sum_{\substack{  a\in \Omega\\ p \mid m_i(a),   \, i \in S'}}
	\chi_B(a)  \\
	&=  \sum_{S' \subset S^c} (-1)^{|S'|}g_{S \cup S'}(p),
\end{align*}
by \eqref{def:gdensity} and \eqref{def:brandemb3}, as required.
The properties \eqref{eqdef:3} and  \eqref{eqdef:4} then follow easily.
Then \eqref{eqdef:1} follows from Kolmogorov’s extension theorem \cite[Thm.~ 2.1.21]{Dur19}. 
\end{proof}
 
For
$1\leq i \leq n $ and $B>1$ define 
  the random variable
\[
S_{i,B}
:=\sum_{
{ A <  }
p\leq
{
\c F(B)}^{\psi(B)}
 }
\theta_i(p)
X_{i,p}
.\] A very special case of the definition of the $X_{i,p}$ is that
\[
\mathrm{Prob}[X_{i,p}=1]=
g_i(p) 
\text{  and  }
\mathrm{Prob}[X_{i,p}=0]=1-g_i(p) 
,\] hence, recalling~\eqref{def:limitations} and~\eqref{eq:Trilogy Suite Op. 5 - Yngwie Malmsteen},
    for all $1\leq i \leq n $ and   $T\geq 0$
we get 
\[
\E\l[S_{i,B}
\r]=
\c M_i(\c F(B)^{\psi(B)}  )
\text{ and }
\Var\l[S_{i,B}
\r]^{1/2}
=\c V_i(\c F(B)^{\psi(B)})
.\]  
In verifying the last two equalities we have implicitly used that $g_i(p)=0$
for $p\leq A$, as can be seen by Definition~\ref{def:g_i}.
In our 
next lemma
we use 
Lemma~\ref{lem:31billmoza},
that regards  moments 
of 
\[\sum_{
{  A <   }
p \leq \c F(B)^{\psi(B)}} \theta_i(p)  
\mathds 1 _{p \Z} (m_i(a))
,\] to
study the moments 
 of 
\[\sum_{
 {  A <   }
p \leq \c F(B)^{\psi(B)}} \theta_i(p)  
\l(
\mathds 1 _{p \Z} (m_i(a))
-
g_i(p) 
\r)
,\]
which are closer to $\b K (a)$ in~\eqref{defeqk}.

\begin{lemma}
\label
{lem:marcello_bach23}
For
each $\b k \in \N^n$
the following   holds with an absolute implied constant,
\begin{align*}
&
\hspace{-0,4cm}
\mathbb{E}_B\l[
\!
\prod_{i=1}^n
 \l(
 \sum_{
{  A <   }
p \leq 
{\c F(B)}^{\psi(B)}
} 
\hspace{-0,4cm}
\theta_i(p) \mathds 1_{p\Z}(m_i(a))
- \c M_i({\c F(B)}^{\psi(B)})
\r)^{k_i}
\!
\r]\!
\!-\!
\E\l[ 
\prod_{i=1}^n
\l(
S_{i,B} -
\E\l[S_{i,B}\r]
\r)^{k_i}
\!
\r]
 \\
\ll
&
\frac{\prod_{i=1}^n
\l(
|\c M_i(\c F(B)^{\psi(B)})|
+\Theta
\r)^{k_i}}{M(B)}
\l\{
 \sum_{\b d \in \N^n,\eqref{eq:manzacafe}}
 \l(
g(\b d)
|\c R((1,\ldots,1),B)| 
+
|\c R(\b d,B)|
\r)
\r\}
.
\end{align*}
 \end{lemma}
\begin{proof} 
Note that  by~\eqref{eq:::defmult} and~\eqref{eqdef:1}-\eqref{eqdef:3}   we get 
\begin{align*}
\mathbb{E}
\l[
\prod_{i=1}^n
S_{i,B}^{k_i}
\r]
&=\sum_{\substack{
1 \leq i \leq n \\ 1 \leq j \leq k_i \\
A< p_{i,j} \leq \mathcal{F}(B)^{\psi(B)} 
}}
\l\{\prod_{\substack{
1\leq u \leq n \\
1\leq v \leq k_u %ES k_i
}}
\theta_u(p_{u,v}
)\r\}
\mathbb{E}\l[
\prod_{\substack{
1\leq u \leq n \\
1\leq v \leq k_u}}
X_{u,p_{u,v}}\r]
\\
&=\sum_{\substack{
1 \leq i \leq n \\ 1 \leq j \leq k_i \\
A< p_{i,j} \leq \mathcal{F}(B)^{\psi(B)} 
}}
\l\{\prod_{\substack{
1\leq u \leq n \\
1\leq v \leq k_u 
 }}
\theta_u(p_{u,v}
)\r\}
g\l(
\mathrm{rad}\l(\prod_{v=1}^{k_1}
   p_{1,v}  \r)
,\ldots,
\mathrm{rad}\l(\prod_{v=1}^{k_n}   p_{n,v}  \r)\r)
,\end{align*}
therefore
the difference in Lemma~\ref{lem:31billmoza}
equals 
\[
\mathbb{E}_B\l[
\prod_{i=1}^n
 \l(\sum_{
{  A <   }
p \leq {\c F(B)}^{\psi(B)}
}
 \theta_i(p) 
\mathds 1_{p\Z}(m_i(a))
\r)^{k_i}\r]
-
\mathbb{E}
\l[ \prod_{i=1}^n
S_{i,B}^{k_i}
\r]
.\]
Using the binomial theorem 
we see that 
\begin{align*}
&\mathbb{E}_B\l[
\prod_{i=1}^n
 \l(-
\c M_i({\c F(B)}^{\psi(B)})
+\sum_{
{  A <   }
p \leq {\c F(B)}^{\psi(B)}
} \theta_i(p) 
\mathds 1_{p\Z}(m_i(a))\r)^{k_i}\r]
\\
&=
\sum_{ \substack{
0\leq j_1 \leq k_1
\\ 
\ldots 
\\
 0\leq j_n \leq k_n
}}
\l\{
\prod_{i=1}^n
 (-\c M_i({\c F(B)}^{\psi(B)})
 )^{k_i-j_i}
{k_i\choose j_i}
 \r\}
\mathbb{E}_B\l[ 
\prod_{i=1}^n
\l(\sum_{
{  A <   }
p
 \leq {\c F(B)}^{\psi(B)}
} \theta_i(p) 
\mathds 1_{p\Z}(m_i(a))
\r)^{j_i}\r]
\end{align*}
and \[
\mathbb{E}
\l[
\prod_{i=1}^n
\l(-
\E\l[S_{i,B}   \r] 
+S_{i,B}  \r)^{k_i}\r]
=
\sum_{ \substack{
0\leq j_1 \leq k_1
\\ 
\ldots 
\\
 0\leq j_n \leq k_n
}}
\l\{
\prod_{i=1}^n
 (-\c M_i({\c F(B)}^{\psi(B)})
 )^{k_i-j_i}
{k_i\choose j_i}
 \r\}
\mathbb{E}
\l[\prod_{i=1}^n
S_{i,B}^{j_i}\r]
.\]
Alluding to
Lemma~\ref{lem:31billmoza}
  shows that the difference in our lemma is 
\[\ll
\frac{1}{M(B)}
\hspace{-0,1cm}
\sum_{ \substack{
0\leq j_1 \leq k_1
\\ 
\ldots 
\\
 0\leq j_n \leq k_n
}}
\hspace{-0,2cm}
\l\{
\prod_{i=1}^n
|\c M_i({\c F(B)}^{\psi(B)})
|^{k_i-j_i}
{k_i\choose j_i}
\Theta^{j_i}  
 \r\}
\hspace{-0,1cm}
\sum_{\b d \in \N^n} 
\hspace{-0,1cm}
 \l(
g(\b d)
|\c R((1,\ldots,1),B)| 
+
|\c R(\b d,B)|
\r),\]
where the sum over $\b d $ is subject to the same conditions as in~\eqref{eq:manzacafe}, 
except that  $\omega(d_i)\leq k_i$ must be 
 replaced by    $\omega(d_i)\leq j_i$.
Noting that $j_i \leq k_i$ and that each term in the sum over $\b d $ is non-negative, 
we may bound the sum over $\b d $ by the same one where the summation is over those 
$\b d $ that satisfy~\eqref{eq:manzacafe}.
Therefore, the last quantity is at most 
\[
\frac{1}{M(B)}
\hspace{-0,1cm}
\sum_{ \substack{
0\leq j_1 \leq k_1
\\ 
\ldots 
\\
 0\leq j_n \leq k_n
}}
\hspace{-0,2cm}
\l\{
\prod_{i=1}^n
|\c M_i({\c F(B)}^{\psi(B)})
|^{k_i-j_i}
{k_i\choose j_i}
\Theta^{j_i}  
 \r\}
\hspace{-0,1cm}
\sum_{\substack{
\b d \in \N^n
\\
\hspace{-0,1cm}
\eqref{eq:manzacafe}
}} 
\hspace{-0,2cm}
 \l(
g(\b d)
|\c R((1,\ldots,1),B)| 
+
|\c R(\b d,B)|
\r)
.\]
The proof is   concluded by noting  that 
the sum  over $j_i$ is 
$
 \l(|\c M_i({\c F(B)}^{\psi(B)})
|+\Theta \r)^{k_i}
$.\end{proof} 
For the rest of this section we fix an arbitrary 
vector $\b t \in \R^n$.
We are now in place to study the linear combination in~\eqref{eq:kaloe?}
by modelling it via a linear combination of the random variables
$S_{i,B}$.
More specifically, 
for every prime $p > A $ we 
define the 
random variable 
\beq
{def:immolation dawn of possessiom}
{
Y_p:=
\sum_{i=1}^n
\frac{ t_i \theta_i(p)  \l(X_{i,p}-g_i(p)\r) }
{\Var\l[  S_{i,B}   \r]^{1/2}}
.} 
We next reformulate the previous lemmas 
using the variables $Y_p$. 
\begin{lemma}
\label
{lem:finalzeta}
For every $k\in \N$ 
the function of $B$ given by 
\begin{align*}
 \E_B
&
\l[
\l(
\sum_{i=1}^n t_i 
\frac{\l(\sum_{
{  A <   }
p\leq  \c F(B)^{\psi(B)}
} \theta_i(p) 
\mathds 1 _{p \Z} (m_i(a) )
 \r)-\c M_i(\c F(B)^{\psi(B)})}{\c V_i(\c F(B)^{\psi(B)})  }
\r)^k
\r]
\\
 &
\hspace{1,2cm}
-
\E\l[\l(
 \frac{ 
 \sum_{
{  A <   }
p\leq \c F(B)^{\psi(B)}
}Y_p 
 }{\Var\l[  \sum_{
{  A <   }
p\leq \c F(B)^{\psi(B)}
}Y_p  \r]^{1/2}}
\r)^k
\r]
 \end{align*}
tends to $0$ as $B\to \infty$.
 \end{lemma}
\begin{proof}
Using the multinomial theorem we see that the quantity
 $\E_B\l[\cdot \r]$ in the lemma 
equals 
\[
\sum_{\substack{
\b k \in (\Z_{\geq 0})^n\\
k_1+\cdots+k_n=k
}}
\frac{k!
t_1^{k_1}\cdots t_n^{k_n}
}{k_1! \cdots k_n!}
\E_B\l[
\prod_{i=1}^n
\l(
  \frac{\l(\sum_{
p\mid m_i(a), 
{  A <   }
p\leq 
\c F(B)^{\psi(B)}
} \theta_i(p)  \r)-\c M_i(\c F(B)^{\psi(B)})}{\c V_i(\c F(B)^{\psi(B)})  }
\r)^{k_i}
\r]
.\]
By~\eqref{def:immolation dawn of possessiom}
we see that the term $\E\l[\cdot\r]$ in the lemma can similarly be written as 
\[
\sum_{\substack{
\b k \in (\Z_{\geq 0})^n\\
k_1+\cdots+k_n=k
}}
\frac{k!
t_1^{k_1}\cdots t_n^{k_n}
}{k_1! \cdots k_n!}
\E\l[
\prod_{i=1}^n
\l(
 \frac{S_{i,B}-\E\l[S_{i,B}\r]
}{\Var\l[S_{i,B}\r]^{1/2}}
\r)^{k_i}
\r]
.\]
Subtracting the last two equations 
and invoking
Lemma~\ref{lem:marcello_bach23}
and~\eqref{eq:thebigassumptn}
concludes the proof.
\end{proof}

Our plan is   to use 
the Central Limit Theorem to study the distribution of 
$\sum_{p} Y_p$.
Before that   
 we   need to study some basic 
 properties of $Y_p$.
\begin{lemma}
\label
{lem:viv2viol} 
 \hfill
\begin{enumerate}
\item The random variables $Y_p$ are independent;
\item For every prime $p$ we have $\E\l[Y_p\r]=0$;
\item For every prime $p$ the quantity $\Var\l[Y_p\r]$
equals
 \[
\sum_{i=1}^n
\frac{ t_i^2 \theta_i(p)^2  g_i(p) \l(1-g_i(p)\r)  }  {\Var\l[  S_{i,B}   \r]}
+2
\sum_{1\leq i < j \leq n }
\frac{t_i t_j  \theta_i(p)   \theta_j(p) 
\l(g_{\{i,j\}}(p)- g_i(p) g_j(p)  \r)
 }{\Var\l[  S_{i,B}   \r]^{1/2}   \Var\l[  S_{j,B}   \r]^{1/2}  }
;\]
\item We have 
\[
\lim_{B \to+\infty}
\Var\l[\sum_{
{  A <   }
p\leq 
{\c F(B)}^{\psi(B)} 
} Y_p\r]=Q (\b  t ) 
, \quad \mbox{where }
Q (\b  t )
:=
\sum_{i=1}^n
  t_i^2
+2\sum_{1\leq i < j \leq n }
\sigma_{ij}
t_i t_j,
 \] and 
 $\sigma_{ij}$ are given by~\eqref{A. VIVALDI:Filiae maestae Jerusalem RV 638}.
In particular, we have 
$Q (\b  t )\geq 0$.
 \end{enumerate}
\end{lemma}
\begin{proof}
\hfill
\begin{enumerate}
\item This follows  directly from
\eqref{eqdef:1}. 
\item This follows from linearity of expectation and the fact that $\E \l[X_{i,p}  \r]=g_i(p).$
\item Recall that the covariance of two random variables $W_1,W_2$
is defined 
by 
\[
\Cov\l[ W_1 , W_2 \r]
:= \E\l[ W_1  W_2 \r]
-\E\l[ W_1   \r]
\E\l[   W_2 \r]
.\]
Using the standard formula $\Var\l[\sum_{i=1}^n X_i\r]=\sum_i \Var[X_i]+2 \sum_{i<j} \Cov\l[X_i,X_j\r]$
shows that $
\Var\l[Y_p\r]
$  equals 
\begin
{align*}
&\sum_{i=1}^n
\frac{ t_i^2 \theta_i(p)^2  \Var\l[   X_{i,p}-g_i(p)  \r]    }  {\Var\l[  S_{i,B}   \r]}
\\+&2
\sum_{1\leq i < j \leq n }
\frac{t_i t_j  \theta_i(p)   \theta_j(p) 
\Cov\l[
 \l(X_{i,p}-g_i(p)\r)   ,
 \l(X_{j,p}-g_j(p)\r) 
\r]
 }{\Var\l[  S_{i,B}   \r]^{1/2}   \Var\l[  S_{j,B}   \r]^{1/2}  }
.\end{align*}
Using the rules
$\Var[X+c]=\Var[X]$
and 
$
\Cov\l[ X-c ,  Y-c'  \r]
=
\Cov\l[X,Y\r]
$
this becomes 
\[
\sum_{i=1}^n
\frac{ t_i^2 \theta_i(p)^2  \Var\l[   X_{i,p} \r]    }  {\Var\l[  S_{i,B}   \r]}
+2
\sum_{1\leq i < j \leq n }
\frac{t_i t_j  \theta_i(p)   \theta_j(p) \Cov\l[X_{i,p}, X_{j,p} \r]
 }{\Var\l[  S_{i,B}   \r]^{1/2}   \Var\l[  S_{j,B}   \r]^{1/2}  }
.\]
By~\eqref{eqdef:3}-\eqref{eqdef:4} we have 
$\Var\l[   X_{i,p} \r]  = g_i(p) \l(1-g_i(p)\r)
$ and \[
\Cov\l[X_{i,p}, X_{j,p} \r] =\E\l[X_{i,p} X_{j,p}\r] - \E\l[X_{i,p} \r]  \E\l[ X_{j,p}\r]=   
g_{\{i,j\}}(p)- g_i(p) g_j(p) 
,\] which
concludes the proof. 
\item 
 Using the third part of the present lemma shows that 
\[
\Var\l[  \sum_{
{  A <   }
p\leq {\c F(B)}^{\psi(B)} 
} Y_p\r]
=
\sum_{i=1}^n t_i^2
 + 
2
\hspace{-0,2cm}
\sum_{1\leq i < j \leq n }
t_i t_j  
\hspace{-0,2cm}
 \sum_{p\leq {\c F(B)}^{\psi(B)}  }
\hspace{-0,5cm}
\frac{
\theta_i(p) \theta_j(p)
\l(g_{\{i,j\}}(p)- g_i(p) g_j(p)  \r)
}
{\Var\l[  S_{i,B}   \r]^{1/2}   \Var\l[  S_{j,B}   \r]^{1/2}  }
,\]
where we used that if $p\leq A$ then $g_{\{i,j\}}(p)=0=g_i(p)$.
This equals 
\[
\sum_{i=1}^n t_i^2
+2
\sum_{1\leq i < j \leq n }
t_i t_j
 \Cov\l[\frac{S_{i,B}
 -  \E\l[S_{i,B}\r]}{ \Var\l[  S_{i,B}   \r]^{1/2} },\frac{S_{j,B}
 -  \E\l[S_{j,B}\r]}{ \Var\l[  S_{j,B}   \r]^{1/2} }\r]
.\]
One of the assumptions of Theorem~\ref{thm:mainthmr}
is that
the limits in~\eqref{A. VIVALDI:Filiae maestae Jerusalem RV 638}
exist. A direct comparison with the last expression here shows that   \[
\lim_{B\to+\infty}
\Var\l[  \sum_{
{ A <  } 
p\leq {\c F(B)}^{\psi(B)} 
} Y_p\r]=Q(\b t)\] due to~\eqref{eq:Flots du Danube}.
As a consequence, we obtain that
 $Q(\b t)$ is  
non-negative.  \qedhere
\end{enumerate}
\end{proof}
We are now in position to apply the  Central Limit Theorem
to $\sum_p Y_p$. 
\begin{lemma}
\label
{lem:to lima twn fakwn12} 
For all $\b t\in \R^n$ with 
$Q(\b t )> 0$ the sequence of random variables 
$$\frac{1}{Q(\b t)^{1/2}}
\sum_{
{ A <  } 
p\leq \c F(B)^{\psi(B)}}
Y_p
$$ 
converges in distribution
to the standard normal  distribution as $B\to\infty$.
 If 
$Q(\b t)=0$,  
then  
$$ 
\sum_{
{ A <  } 
p\leq \c F(B)^{\psi(B)}}
Y_p
 $$ 
converges in distribution
to  $0$
 as $B\to\infty$.
\end{lemma}
\begin{proof}
The first two parts of Lemma~\ref{lem:viv2viol}
imply that 
\[\E\l[
\sum_{
{ A <  } 
p\leq \c F(B)^{\psi(B)}}
Y_p\r]=0.
\] Therefore, if 
 $Q(\b t)=0$
then 
  the last part of Lemma~\ref{lem:viv2viol},  
implies that 
\[\E\l[\l(
\sum_{
{ A <  } 
p\leq \c F(B)^{\psi(B)}}
Y_p\r)^2
\r]
=
\Var\l[
\sum_{
{ A <  } 
p\leq \c F(B)^{\psi(B)}}
Y_p\r]
\to 
0, \text{ as } B\to+\infty.
\] Chebyshev's inequality then
yields $\sum_p Y_p \Rightarrow 0$.

We now 
assume that $Q(\b t)> 0$.
It is clear from   
the last part of Lemma~\ref{lem:viv2viol}
that we only have to show that one can apply the Central Limit Theorem to the sum $\sum_p Y_p$ 
in the present lemma.
To do this we shall verify that  Lindeberg's
condition for 
the 
Central Limit Theorem for
triangular arrays~\cite[Theorem 27.2]{MR1324786}
 is satisfied. 
Recalling that  $\E\l[Y_p\r]=0$,
this condition
can be written as 
 \[
\lim_{B\to+\infty}
\frac{1}{\Var\l[\sum_{
{ A <  } 
p\leq \c F(B)^{\psi(B)}}
 Y_p\r]}
\sum_{
{ A <  } 
p\leq \c F(B)^{\psi(B)}}
\E\l[
 Y_p^2
\mathds 1\l(
\l \{  |Y_p|>\delta \Var\l[\sum_p Y_p\r]   \r\}
\r)
\r]
=0
.\]
By
the last part of Lemma~\ref{lem:viv2viol}
it is clear 
that this is
equivalent to showing that 
for all $\delta>0$
\beq
{eq:lyapnassumption}
{
\lim_{B\to+\infty} 
\sum_{
{ A <  } 
p\leq \c F(B)^{\psi(B)}}
\E\l[
 Y_p^2
\mathds 1
\l(  \l  \{|Y_p|>\delta \r \}   \r )
\r]
=0
.} 
To prove~\eqref{eq:lyapnassumption}, 
note by the definition~\eqref{def:immolation dawn of possessiom}
and the bounds~\eqref{assumption:thet},
$g_i(p)\leq 1$,
we obtain 
\beq
{eq:Dead Congregation - Martyrdoom 2012 }
{
|Y_p|\ll_{\delta,n,\b t}
 \frac{1}{\min_i 
\Var\l[  S_{i,B}   \r]^{1/2}
}
,} where the implied constant
is independent of $p$ and $B$.
Hence, for any
 fixed   $\delta>0$
we see that
by
assumption~\eqref
{eq:asumptn}
we have 
 $\mathds 1( \{|Y_p|>\delta \}  )=0$ 
for all sufficiently large $B$.
This is sufficient for~\eqref{eq:lyapnassumption}.
   \end{proof}

Lemma~\ref{lem:finalzeta} shows that the moments of the number-theoretic 
objects  
 $\sum_p \theta_i(p) \mathds 1_{p \Z} (m_i(a) ) $
essentially behave like the moments of certain random variables related to  
  $ \sum_p Y_p $,
and in 
Lemma~\ref{lem:to lima twn fakwn12} 
we saw that   $  \sum_p Y_p $
has  
a limiting distribution. 
To pass from this to
limiting distributions for the number-theoretic objects 
we first 
need to prove certain growth estimates for the moments of the related random variables.
This is the goal  
of
 the next lemma.
\begin{lemma}
\label
{Socrates Drank the Conium - Taste of Conium}
Assume that $\b
t \in \R^n$ is fixed and that 
$Q(\b t)>0$.
Then there exists  $L>0$ (that is independent of $B$ and $k$)
such that for all $k\in \N$ and $B\geq 1$
one has 
\[\l|
\E\l[\l(
 \frac{ 
 \sum_{
{ A <  } 
p\leq 
{\c F(B)}^{\psi(B)} 
}Y_p 
 }{\Var\l[  \sum_{
{ A <  } 
p\leq 
{\c F(B)}^{\psi(B)} 
}Y_p  \r]^{1/2}}
\r)^k
\r]
 \r|
\ll
k!
L^k
,\] where the implied constant is independent of $B, k$ and $L$.
\end{lemma}
\begin{proof} By the last part of Lemma~\ref{lem:viv2viol} 
we have 
$\lim_{B}
\Var[\sum_{p} Y_p]
=Q(\b t)$, 
therefore 
there exists $B_0\geq 0$ such that $\Var\l[\sum_{p } Y_p\r]$ is strictly positive for all $B\geq B_0$.
For such $B$ we 
let 
\[Z_p:=
 Y_p
 \Var\l[  \sum_{
{ A <  } 
p\leq {\c F(B)}^{\psi(B)} }Y_p  \r]^{-1/2},
 \]
so that (recall that the $Y_p$ are independent by the first part of Lemma \ref{lem:viv2viol})
\beq
{eq:pizasupataper}
{
\hspace{-0,2cm}
\E\l[\l(
 \frac{ 
  \sum_{
{ A <  } 
p\leq {\c F(B)}^{\psi(B)} }Y_p 
 }{\Var\l[  \sum_{
 { A <  } 
p\leq {\c F(B)}^{\psi(B)} }Y_p  \r]^{1/2}}
\r)^k
\r]
 =
 \sum_{\substack{
1\leq u \leq k \\
(k_1,\ldots,k_u) \in \N^u\\
k_1+\cdots+k_u=k
}}
\frac{k!}{
\prod_{i=1}^u k_i!
}
\Osum_{
}
\prod_{i=1}^u
\E\l[Z_{p_i}^{k_i}  \r]
,} where the sum $\Osum$ is over prime
tuples satisfying 
$A <  p_1<\ldots<p_u \leq  {\c F(B)}^{\psi(B)} $.
 Note that we have $\E\l[Z_p\r]=0$,
therefore we can
add the restriction that every $k_i$ is strictly larger than $1$.
By the bound~\eqref{eq:Dead Congregation - Martyrdoom 2012 }
we deduce  that
there exists $\c L   $ such that 
 for all $B\geq 1$
one has 
$|Z_p|\leq \c L \leq 1+\c L
$. Therefore, for all $k_i\geq 2$
we have 
$$|\E\l[Z_p^{k_i}\r]|
\leq \c (1+\c L)^{k_i-2}   \E\l[Z_p^2\r]
\leq \c (1+\c L)^{k_i}   \E\l[Z_p^2\r]
.$$
Thus, using $k_1+\cdots +k_u=k$, we obtain 
\[\sum_{
 { A <  } 
p_1<\ldots<p_u \leq {\c F(B)}^{\psi(B)} 
}
\prod_{i=1}^u 
\E\l[Z_{p_i}^{k_i}  \r]
  \leq (1+\c L)^k
\sum_{
 { A <  } 
p_1<\ldots<p_u \leq 
{\c F(B)}^{\psi(B)} 
}
\prod_{i=1}^u 
\E\l[Z_{p_i}^{2}  \r]
,\]
which is at most 
\[\frac{(1+\c L)^k}{u!}
\l(
\sum_{
 { A <  } 
p\leq 
{\c F(B)}^{\psi(B)} 
}
 \E\l[Z_{p}^2  \r]
\r)^u
.\] Now note that 
\[
\sum_{
 { A <  } 
p\leq 
{\c F(B)}^{\psi(B)} 
} \E\l[Z_{p}^2  \r]
=\frac{1}{\Var\l[\sum_{
 { A <  } 
p\leq 
{\c F(B)}^{\psi(B)} 
}Y_p\r]}
\sum_{
 { A <  } 
p\leq 
{\c F(B)}^{\psi(B)} 
}
\Var\l[Y_p\r]=1
.\] Thus, by~\eqref{eq:pizasupataper}
we get 
\begin{align*}
&\l|
\E\l[\l(
 \frac{ 
 \sum_{
 { A <  } 
p\leq 
{\c F(B)}^{\psi(B)} 
}Y_p 
 }{\Var\l[  \sum_{
 { A <  } 
p\leq 
{\c F(B)}^{\psi(B)} 
}Y_p  \r]^{1/2}}
\r)^k
\r]
\r|
\\
&
\leq 
(1+\c L)^k
\sum_{1\leq u \leq k/2   }
\frac{ 1}{u!}
\sum_{\substack{(k_1,\ldots,k_u) \in (\N_{\geq 2 })^u\\
k_1+\cdots+k_u=k
}}
\frac{k!}{k_1!\cdots k_u!}
\\
&
\leq 
(1+\c L)^k
k!
\sum_{1\leq u \leq k/2   }
\frac{ k^u }{u!}
\leq 
 (1+\c L)   ^{k }
k!
  \mathrm{e} ^{k }
.\end{align*}
 This concludes the proof.
\end{proof}

\subsubsection{
Conclusion of the proof of Theorem~\ref{thm:mainthmr}
}
\label
{s:prokoviev}

The following is the version of the method of moments we shall be using.
It is the reason we proved  Lemmas \ref{lem:finalzeta}, \ref{lem:to lima twn fakwn12}, and
\ref{Socrates Drank the Conium - Taste of Conium}.
\begin{lemma}
[Billinglsey, {\cite[Thm.~11.2]{MR0466055}}]
\label
{lem:billbill}
Let
$\zeta, \xi_n, \zeta_n, (n \in \N),$ be random variables. Suppose that 
$ \zeta_n\Rightarrow \zeta$ and that
$$\lim_{n\to+\infty} 
\l|
\E\l[\xi_n^k\r]
-
\E\l[\zeta_n^k\r]
\r|=0 \text{ for all }k \in \N \quad \text{and} \quad
\sup_{n,k \in \N}
\frac{\l|  \E\l[\zeta_n^k\r] \r|}{k! L^k}\leq 1 \text{ for some }L\in \R.$$
Then 
$\xi_n \Rightarrow \zeta$.
\end{lemma}
Note that by the fourth part of Lemma~\ref{lem:viv2viol}
the matrix $\boldsymbol \Sigma$ defined in Theorem~\ref{thm:mainthmr}
is positive semi-definite, so
the multivariate normal distribution 
$\c N(\b 0, \boldsymbol \Sigma)$
is well-defined.
Let $\b X$ be a random vector in $\R^n$
with distribution 
$\c N(\b 0, \boldsymbol \Sigma)$.
Using the Cram\'er--Wold theorem~\cite[Thm.~29.4]{MR1700749}
we see that the convergence of the sequence
\eqref{eq:thelimit} to a multivariate normal distribution
 follows if we show that for every $\b t \in \R^n$ 
one has 
\[
\sum_{i=1}^n t_i 
\frac{\l(\sum_{p\mid m_i(a)} \theta_i(p)  \r)-\c M_i(m_i(a))}{\c V_i(m_i(a))}
\Rightarrow 
\sum_{i=1}^n t_i X_i
.\]
In light of Lemma~\ref
{lem:albinionioboe7}
this is equivalent to proving \[
\sum_{i=1}^n t_i 
\frac{\l(\sum_{
 { A <  } 
p\mid m_i(a), p\leq 
\c F(B)^{\psi(B)}
} \theta_i(p)  \r)-\c M_i(\c F(B)^{\psi(B)})}{\c V_i(\c F(B)^{\psi(B)})  }
\Rightarrow 
\sum_{i=1}^n t_i X_i
.\]
This can be deduced by
injecting
Lemmas~\ref{lem:finalzeta},
~\ref{lem:to lima twn fakwn12}
and~\ref{Socrates Drank the Conium - Taste of Conium}
 into
Lemma~\ref{lem:billbill}. \qed

\subsection{
The proof       of 
Theorem~\ref{thm:mainthmrqw}
}
\label
{s:gamwtoxristosoueasy}

The proof is a
combination of the Fundamental Lemma of the Combinatorial Sieve
and 
Theorem~\ref{thm:mainthmr}.
As a first step we show that the function $\c F$ tends to infinity.
\begin{lemma}
\label
{lem:howtouseprimes}
In the setting of Theorem~\ref{thm:mainthmrqw}
we have 
$\lim_{B\to \infty}
\c F(B)=+\infty$. 
\end{lemma}\begin{proof}
Let $1 \leq i \leq n$.
By~\eqref{eq:asumptnqwe}
there are infinitely many primes $p$ with $g_i(p) >0$.
Let $p$ be such a prime. By~\eqref{def:gdensityqwe}
with $ \b d=(
p
\mathds 1 _{\{i\}}(j)
+
\mathds 1 _{\{1,\ldots, n \}\setminus \{i\}}(j))_{j=1}^n $
we get 
\[
\liminf_{B\to+\infty} 
 \frac{\#\l\{
a\in \Omega: h(a) \leq B,
 p\leq m_i (a) 
\r\}}{
N(B)
} >0
.\] Thus, by~\eqref{def:northcot2hgf8sqwe}
we obtain $ p\leq \liminf_{B\to+\infty} \c F(B)$. Taking $p \to \infty$
concludes the proof.
\end{proof}
Before proceeding  
we must show 
that 
  the typical size of $ m_i(a)$ is not
too small. This 
  will require the 
Fundamental Lemma of the Combinatorial Sieve as given 
in~\cite[Thm.~3, p.~60]{MR1342300}.  
\begin{lemma} 
\label 
{lem:upperboundsieve}
Let $\c A$ be a finite 
set of integers and $\mathfrak{P}$ a set of primes.
If there exist   real numbers $A,\kappa>0$ 
and a multiplicative function $g:\N \to \R \cap [0,\infty)  $ such that 
\beq
{eq:brunsassumption}
{
\prod_{\eta \leq p \leq \xi } 
\l(1-\frac{g(p) 
}{p}
\r)^{-1}
\leq
\l(\frac{\log \xi}{\log \eta} \r)^\kappa
\l( 1+\frac{A}{\log \eta }\r), \,\,
(\forall \, 2\leq \eta  \leq \xi ),  
}
then
for all $X,y,u>1$
 the number of 
 $a\in \c A$ that are coprime to every   $p 
\in \mathfrak P
\cap
(0,y ]
$
is \[
\ll X
 \prod_{\substack{ p\in \mathfrak P\\ 
p \leq y 
  }}
\l(1-\frac{g(p)}{p}\r)
+
\sum_{\substack{ d \leq y^u
\\ 
p\mid d \Rightarrow p\in \mathfrak P}}
\mu(d)^2
 \l|
\#\l\{a\in \c A: d\mid a 
 \r\}
-
\frac{g(d)}{d} X
\r|
,\] where
the implied constant depends at most on $\kappa$ and $A$.
\end{lemma}

\begin{lemma}
\label
{lem:undertheasu}
Let  
  $\epsilon$ be as in~\eqref{def:northcot2hgf8sqwerwe5we}
and let $
\epsilon_1:[1,\infty)
\to
[0,\infty)$ satisfy
\beq{eq:kanothxari}
{
\lim_{B\to \infty} 
\epsilon_1(B)
=
0.
}
Define 
$
z_0(B):=
\c F(B)^{  \epsilon_1(B)  \epsilon(B)    }
$
and assume that $
\lim_{B\to \infty} 
z_0(B)
=
\infty$.  
Then 
\[\lim_{B\to \infty} 
\b P_B\l[a\in \Omega:
z_0(B)
\geq 
\min_{1\leq i \leq n}
m_i(a)
\r]
=0
.\]
\end{lemma}
\begin{proof}
By Boole's inequality it is sufficient to show that for all ${1\leq i \leq n}$
one has \beq
{eq:sigatikitrela}
{
\lim_{B\to \infty} 
\b P_B\l[a\in \Omega:
z_0(B)
\geq 
 m_i(a)
\r]
=0
.} 
Let 
$
z(B):=
\c F(B)^{  \epsilon(B)/2 }
$.
To prove~\eqref{eq:sigatikitrela}
 we 
note
 that the inequality 
$z_0(B)
\geq 
 m_i(a)$
implies that for every prime $p \in (z_0(B), z(B) ) $ 
we have 
$p\nmid m_i(a)$.
Therefore, letting $W:=\prod_{z_0(B)< p <  z(B)  }p$, 
we get 
\[
\#\{ 
 a\in \Omega:
 h(a)\leq B  ,
z_0(B) \geq   m_i(a)
\}
\leq  \#\{ a\in \Omega, h(a)\leq B ,
\gcd(m_i(a), W)=1
\}
 .\]
 We   use 
Lemma~\ref{lem:upperboundsieve}
with 
$
\c A:= \{m_i(a): a\in \Omega, h(a) \leq B\}$ and 
$$\mathfrak{P}:=\{p> z_0(B): p \text{ prime}  \}, 
X:= N(B),
\kappa:=c_i,
g(p):=p 
g_i(p),
u :=2,
y:=z(B)-1.
  $$  
Assumption~\eqref{eq:asumptnqwe}
and the estimate $\log (1-z)^{-1} =z+O(z^2), |z|<1$ 
show that 
\begin{equation}
\label
{eq:carbomb_centralia}
\log 
\prod_{\eta\leq p \leq \xi} (1-g_i(p) )^{-1} 
=
\sum_{\eta\leq p \leq \xi}
g_i(p)
+O\l(
\sum_{\eta\leq p \leq \xi}
g_i(p)^2
\r)
=c_i \log \l(\frac{ \log \xi }{\log \eta} \r)
+O\l(\frac{1}{\log \eta} \r)
,\end{equation} 
from which we infer~\eqref{eq:brunsassumption} 
by using $\exp(\epsilon)= 1+O(\epsilon) $
 for   $\epsilon=O( 1/\log \eta)$.  
Lemma~\ref{lem:upperboundsieve} gives the following 
bound  for 
$\#\{ a\in \Omega, h(a)\leq B ,
\gcd(m_i(a), W)=1
\}$,
\[
 \ll 
N(B)
 \prod_{
z_0(B)< p  \leq z(B) 
}
 (
1-
g_i (p) 
)
+ \hspace{-5pt}
\sum_{\substack{ 
d \leq z(B)^2
\\ 
 p\mid d \Rightarrow p> z_0(B)   
 }}
 \hspace{-5pt}
\mu(d)^2
 \l|
\#\l\{a\in \Omega: h(a) \leq B, 
d\mid m_i(a) 
 \r\}
-
g_i(d)
N(B)
\r|
.\] 
Exponentiating~\eqref{eq:carbomb_centralia}
shows that 
the first term is 
\[\ll
N(B)
\l( \frac{\log z_0(B)} {\log z(B)}\r)^{c_i } 
\ll
N(B)
\epsilon_1(B)^{c_i } 
=o(N(B) ).\]
To bound the second term 
we note that $z_0(B)>A$ for $B$ sufficiently large due to the assumption $\lim_{B\to \infty} 
z_0(B)
=
\infty$.
Therefore, we can replace the condition $p>z_0(B)$ by $p>A$, thus, 
 the second term is 
 \[\ll  
\sum_{\substack{ 
d \leq z(B)^2
\\ 
 p\mid d \Rightarrow p> A 
 }}
\mu(d)^2 
|\c R((1,\ldots , 1, d , 1 \ldots, 1), B) |
,\] 
where every component in 
the vector in $\c R$ 
equals $1$, except the $i$-th 
entry, which equals $d$.  By~\eqref{eq:thebigassumptnqwe}
 and Lemma~\ref{lem:howtouseprimes}
we 
immediately
find that this is 
$  o(N(B))$.  
This verifies~\eqref{eq:sigatikitrela}
and hence
concludes the proof.
 \end{proof}
We will use 
Theorem~\ref{thm:mainthmr}
with
\[
\chi_B(a):= \mathds 1_{[0,B]}(h(a)), \,
\theta_i(p)=1
\text{ and }
M(B)=N(B)
.\]
 With these choices we
see that~\eqref{def:northcot2},
~\eqref{def:northcot3}
and~\eqref{def:mmde}
are satisfied due to~\eqref{def:northcot2qwe}.
The assumption~\eqref{assumption:thet}
obviously
holds with $\Theta=1$. We next show that $(\c F, \psi)$ fulfils
the truncation-pair
Definition~\ref{def:tranc},
where $\c F$ is
as in~\eqref{def:northcot2hgf8sqwe} and 
\beq
{defepsqwr}
{\psi(B):=
\frac{\epsilon(B)}{  \sqrt{   \log \log \log \c F(B) } }
= \frac{\sqrt{   \log \log \log \c F(B) } }{  \sqrt{   \log \log \c F(B) } }
,}
where the equality is by  \eqref{def:northcot2hgf8sqwerwe5we}.
Firstly,~\eqref{def:northcot2hgf8s}
follows directly from~\eqref{def:northcot2hgf8sqwe}.
Secondly, to verify~\eqref{eq:Flots du Danube}
it is clearly sufficient to show that 
\[
\lim_{B\to+\infty}\psi(B)
\log  \c F(B)=+\infty
.\] This, however, follows by~\eqref{defepsqwr}
and
Lemma~\ref{lem:howtouseprimes}.

Before proceeding,  note that 
by~\eqref{eq:asumptnqwe}
we have 
\begin{align*}
\c M_i(\c F(B)^{\psi(B)})
=&
c_i \log \log \c F(B)
-c_i \log \frac{1}{\psi(B)}
+O(1)
 \\
=&
c_i \log \log \c F(B)
+O(\log \log  \log \c F(B) ),
\end{align*}
where the last estimate is due to
Lemma~\ref{lem:howtouseprimes}
and~\eqref{defepsqwr}.
We similarly have 
 $$\c V_i(\c F(B)^{\psi(B)})^2 =
c_i \log \log \c F(B)
+O(\log \log  \log \c F(B) ).$$ 
The first part of 
\eqref
{eq:Albioni - Oboe Concerto in D minor, op.9 no.2
1.  I  Allegro e non presto}
follows from
\[
\frac{1  }{\psi(B)
\c V_i(\c F(B)^{\psi(B)})
}
 \ll
\frac{1}{\psi(B) \sqrt{ \log \log \c F(B)  } }
=
\frac{ 1 }{ \sqrt{ \log \log \log \c F(B)  } }
=o(1)
,\] which goes to $0$ as $B\to \infty$ by
Lemma~\ref{lem:howtouseprimes}.
To verify the
second   part of 
\eqref
{eq:Albioni - Oboe Concerto in D minor, op.9 no.2
1.  I  Allegro e non presto}
we use 
Lemma~\ref{lem:undertheasu}
with 
\[
\epsilon_1(B)=
  (  \log \log \log \c F(B) )^{-1/2} 
.\] This choice shows that the function $z_0(B)$ of 
Lemma~\ref{lem:undertheasu} 
coincides with $\c F(B)^{\psi(B)}$. By Lemma~\ref{lem:undertheasu}
we deduce that 
there exists 
a set  $S\subset \Omega$ 
with
$\lim_{B\to \infty} 
\b P_B[S] =1$
and such that whenever 
  $a\in  S$ then we have  
for all $1\leq i \leq n $ that 
\beq
{eq:sismosneamakri}
{
\c F(B)^{\psi(B)}
\leq  
 m_i(a)
\leq 
\c F(B)
.}
Hence, for $a \in S$ we get 
\begin{align*}
0
&\leq 
\frac
{\c M_i(m_i(a))-\c M_i(\c F(B)^{\psi(B)}  )  }
{
\c V_i(\c F(B)^{\psi(B)})
}
\leq 
\frac
{\c M_i(\c F(B) )-\c M_i(\c F(B)^{\psi(B)}  )  }
{
\c V_i(\c F(B)^{\psi(B)})
}
\\
&=
\frac
{c_i (\log \log \c F(B) ) +O(1) -c_i (\log \log \c F(B) ) +O(\log \log \log \c F(B)) }
{
\sqrt{c_i  (\log \log \c F(B) )+O(1) }
},
\end{align*}
which, by Lemma~\ref{lem:howtouseprimes},
tends to $0$ as $B\to\infty$.
We are thus left with verifying
~\eqref{eq:assumptionalbinioni}
and~\eqref{eq:thebigassumptn}.
For the former we observe that  
for $a \in S$ we have the following by~\eqref{eq:sismosneamakri},
\[
1
\leq 
\frac{ \c V_i(  m_i(a) ) }{\c V_i(\c F(B)^{\psi (B) }  )}
\leq 
\frac{ \c V_i( \c F(B)  )  }{\c V_i(\c F(B)^{\psi (B) }  )}
=\frac{ \sqrt{ c_i  \log \log \c F(B)+O(1)} }{\sqrt{c_i \log \log \c F(B)  + O(\log \log \log \c F(B))}}
 ,\] which, by Lemma~\ref{lem:howtouseprimes}, tends to $1$ as $B\to \infty$.
We are left with verifying~\eqref{eq:thebigassumptn}. 
Note that 
in our setting one has 
$\c R((1,\ldots,1),B) =0$ due to $M(B)=N(B)$
and~\eqref{def:levelofdistribvcfdgtqwe}, 
hence we only have to show 
\beq
{eq:yngwqerty}
{ \lim_{B\to+\infty}
 \frac{\prod_{i=1}^n \l(1+
|\c M_i(\c F(B)^{\psi(B)})|
   ^{k_i}\r)}{N(B)\prod_{i=1}^n
\c V_i(\c F(B)^{\psi(B)})^{k_i}
}
   \sum_{\substack{
\b d \in \N^n
\\
\!
\!
\eqref{eq:manzacafe}
}} 
|\c R(\b d,B)|
  =0
}
in order to
verify~\eqref{eq:thebigassumptn}.  
Note that the bounds 
\[
\c M_i(\c F(B)^{\psi(B)})
\ll
\log \log \c F(B)
\text{ and }
\c V_i(\c F(B)^{\psi(B)})^2 
\gg 
 \log \log \c F(B)
\]
imply that 
\[ \frac{\prod_{i=1}^n
\l(
|\c M_i(\c F(B)^{\psi(B)})|
  \r)^{k_i}}{ \prod_{i=1}^n
\c V_i(\c F(B)^{\psi(B)})^{k_i}
}
\ll 
(\log \log \c F(B) )^{\frac{k_1+\cdots +k_n}{2}} 
.\]
Now note that the $d_i$
in~\eqref{eq:yngwqerty}
satisfy for all sufficiently large 
$B$,
\[d_i \leq 
\c F(B)^{k_i \psi(B)}
\leq 
\c F(B)^{\epsilon(B)}
,\]
therefore,
\begin{align*}
& \frac{\prod_{i=1}^n
\l(
|\c M_i(\c F(B)^{\psi(B)})|
 +1 \r)^{k_i}}{N(B)\prod_{i=1}^n
\c V_i(\c F(B)^{\psi(B)})^{k_i}
}
   \sum_{\substack{
\b d \in \N^n
\\
\!
\!
\eqref{eq:manzacafe}
}} 
|\c R(\b d,B)|
 \\
\ll
&\frac{(\log \log \c F(B) )^{\frac{k_1+\cdots +k_n}{2}}
}{N(B)}
  \sum_{\substack{
\b d \in \N^n
\\
|\b d |\leq   \c F(B) ^{\epsilon(B)}
\\
p\mid d_1 \cdots d_n 
\Rightarrow p>A
}}
\mu(d_1)^2
\cdots
\mu(d_n)^2 
|\c R(\b d,B)|
,\end{align*}
which is $o(1)$ as can be seen by taking 
$\gamma=1+\frac{1}{2}(k_1+\cdots +k_n)$
in~\eqref{eq:thebigassumptnqwe}.
This confirms~\eqref{eq:yngwqerty}.
 Having verified all assumptions of Theorem~\ref{thm:mainthmr}, 
the result is that the random vector
 \[
\l(
\frac{\omega(m_1(a)) -\sum_{p \leq m_1(a)} g_1(p)   }
{\l( \sum_{p \leq m_1(a)} g_1(p)  (1-g_1(p))\r)^{1/2} }
,\ldots,
\frac{\omega(m_n(a))-\sum_{p \leq m_n(a)} g_n(p)   }
{\l ( 
\sum_{p \leq m_n(a)} g_n(p)  (1-g_n(p))   \r)^{1/2} }
\r)
\]
has the 
limiting distribution as in Theorem~\ref{thm:mainthmrqw}. 
We next  deduce  the analogous distribution
result for the   function given in~\eqref{defeqkqwe}. We  define the random vectors on $\Omega$ by 
\begin{equation*}
    \mathbf{V} = \left( \frac{\mathcal{V}(m_1(a))}{\sqrt{c_1 \log \log \mathcal{F}(B)}}, \ldots,   \frac{\mathcal{V}(m_n(a))}{\sqrt{c_n \log \log \mathcal{F}(B)}} \right), 
\end{equation*}
\begin{equation*}
    \mathbf{T} = \left( \frac{\omega(m_1(a)) - \mathcal{M}(m_1(a))}{\mathcal{V}(m_1(a))}, \ldots,    \frac{\omega(m_n(a)) - \mathcal{M}(m_n(a))}{\mathcal{V}(m_n(a))} \right),
\end{equation*}
and
\begin{equation*}
    \mathbf{M} = \left( \frac{\mathcal{M}(m_1(a)) - c_1 \log \log \mathcal{F}(B)}{\mathcal{V}(m_1(a))}, \ldots,   \frac{\mathcal{M}(m_n(a)) - c_n \log \log \mathcal{F}(B)}{\mathcal{V}(m_n(a))} \right).
\end{equation*}
Recalling \eqref{defeqkqwe}, we  have 
\begin{equation*} 
\mathbf{K} = \mathbf{V} (\mathbf{T} + \mathbf{M})
\end{equation*}
where the product is taken coordinate-wise.
\autoref{thm:mainthmr} implies that $\mathbf{T} \Rightarrow \mathcal{N}(\mathbf{0}, \Sigma)$. Moreover by \eqref{eq:asumptnqwe}, the fact that $\mathcal{F}(B) \to \infty$ and \eqref{eq:sismosneamakri}, we have $\mathbf{M} \Rightarrow \mathbf{0}$ and $\mathbf{V} \Rightarrow \mathbf{1}$ (where $\mathbf{1}$ is the $n$-dimensional vector all of whose coordinates are $1$).
Slutsky's theorem therefore implies that $\mathbf{K} \Rightarrow \mathcal{N}(\mathbf{0}, \Sigma)$.
Furthermore, the limit in~\eqref{A. VIVALDI:Filiae maestae Jerusalem RV 638}
becomes~\eqref{eqn:gij}. This is due to~\eqref{eq:asumptnqwe},
which ensures that  
\[
\sum_{p\leq T
}
\theta_i(p) \theta_j(p)
 g_i(p) g_j(p)  =
\sum_{p\leq T
} 
 g_i(p) g_j(p)  
\leq \l( \sum_{ p_1,p_2 \leq T}
g_i(p_1)^2
g_j(p_2)^2  
\r)^{1/2}
=O(1),
 \] 
by Cauchy--Schwarz, and 
\begin{equation} \tag*{\qed}
\c V_i(T)^2= 
\sum_{p\leq  T}   
 g_i(p)   \l(1-g_i(p)    \r) 
=
\sum_{p\leq T} g_i(p)
+O(1). 
\end{equation}

\section{Application to integral points} \label{sec:application}
In this section we prove Theorem \ref{thm:main} using Theorem \ref{thm:mainthmrqw}.
\subsection{Inner product of divisors}\label{sec:inner_product}
Let $X$ be an integral Noetherian scheme and $\Div X$ the free abelian group generated by the integral (Weil) divisors on $X$.
To simplify some of the statements and proofs in what follows, we introduce an inner product on $\langle \cdot, \cdot \rangle$ on $\Div X$ as follows. For integral divisors $D,E$ we define
$$\langle D, E \rangle = 
\begin{cases}
	1, \quad D = E, \\
	0, \quad D \neq E.
\end{cases}$$
As the integral divisors form a basis of $\Div X$, this extends to an inner product on $\Div X$. Explicitly $\langle D, E \rangle$ is the number of common irreducible components of $D, E \in \Div X$ counted with multiplicity. This extends to $\Div_\RR X := (\Div X) \otimes_\ZZ \RR$ and we let $\| \cdot \|: \Div_\RR X \to \RR_{\geq 0}$ be the induced norm. Our inner product is a convenient piece of notation which should not be confused with more subtle geometric information like intersection numbers of divisors.

\subsection{Points over finite fields}

In the statement, we implicitly only sum over those $p$ with $\mathcal{X}(\F_p) \neq \emptyset$. (A similar convention applies to Corollary \ref{cor:D_i}.)

\begin{proposition} \label{prop:Serre}
	Let $X$ be a geometrically integral variety over $\QQ$ of 
	dimension $n$ and
	$\mathcal{X}$ a model of $X$ over $\ZZ$.
	Let $Z \subsetneq X$ be a reduced  closed subscheme of pure dimension
	$d$ and $\mathcal{Z}$ its closure in $\mathcal{X}$. 
	\begin{enumerate}
	\item We have
	$$\sum_{p \geq T} \left(\frac{\#\mathcal{Z}(\FF_p)}{\# \mathcal{X}(\FF_p)} \right)^2 \ll \frac{1}{T }.$$
	\item We have
	$$\sum_{p \leq T} \frac{\#\mathcal{Z}(\FF_p)}{\# \mathcal{X}(\FF_p)} =
	\begin{cases}
		C_Z + O(1/T),& \mbox{if } d < n - 1, \\
		\langle Z , Z \rangle \log \log T + C'_Z + O(1/\log T), & \mbox{if } d = n-1.
	\end{cases}$$
	for some $C_Z>0$ and $ C_Z' \in \mathbb R$.
	\end{enumerate}
\end{proposition}
\begin{proof}
	By the Lang--Weil estimates \cite{LW54} we have
	$$\#\mathcal{Z}(\FF_p) \ll p^d \ll p^{n-1},  \quad
	\#\mathcal{X}(\FF_p) = p^n + O(p^{n-1/2}).$$
	Then (1) follows  from the estimate $$\sum_{p\geq T } \l( \frac{\#\mathcal{Z}(\FF_p)}{\# \mathcal{X}(\FF_p)} \r)^2 \ll \sum_{p\geq T } p^{-2} \ll \frac{1}{T},  $$ 
	while,  the case $d < n-1$ of (2)    follows from 
	$$ \sum_{p\geq T } \frac{\#\mathcal{Z}(\FF_p)}{\# \mathcal{X}(\FF_p)} \ll \sum_{p\geq T } \frac{p^{d} }{p^n } 
	\leq \sum_{p\geq T } p^{-2} \ll \frac{1}{T}.$$
	It thus suffices to prove (2) when $d = n-1$. 
	For a number field $k$ we denote by
	$z_p(k)$ the number of prime ideals of $k$ of degree $1$ over $p$.
	Let $I$ be the set of irreducible components of $Z$, and for each
	$i \in I$ let $k_i$ be the algebraic closure of $\QQ$
	in the function field of the corresponding irreducible component;
	this is a number field. For all sufficiently large primes $p$,
	the irreducible components of $\mathcal{Z}_{\FF_p}$ which 
	are geometrically integral correspond exactly to those
	prime ideals of $k_i$ of degree $1$ over $p$. Moreover the components
	which are not geometrically integral contain $O(p^{n-2})$ points
	over $\FF_p$, by Lang--Weil.
	Thus applying the Lang--Weil estimates  to each
	irreducible component of $\mathcal{Z}_{\FF_p}$ gives
	$$\#\mathcal{Z}(\FF_p) = \sum_{i \in I}z_p(k_i)p^{n-1} + O(p^{n-3/2})$$
	and hence 
	$$\frac{\#\mathcal{Z}(\FF_p)}{\# \mathcal{X}(\FF_p)} = 
	\sum_{i \in I}\frac{z_p(k_i)}{p} + O\left(\frac{1}{p^{3/2}}\right).$$
	However turning this into a sum over the non-zero prime ideals 
	of the number field, we have
	$$\sum_{p \leq T} z_p(k) = \sum_{\substack{N(\mathfrak{p}) \leq T \\ N(\fp) \text{ prime}}}1 = \sum_{\substack{N(\mathfrak{p}) \leq T }}1 + O(T^{1/2})
	= \Li(T) + O(T\exp(-c\sqrt{\log T}))$$	
	for some constant 
	$c > 0$, where the second equality is by 
	\cite[Lem.~9.3]{Ser12} and the last
	by the prime ideal theorem \cite[Thm.~3.1]{Ser12}.
	The result now follows from partial summation.
\end{proof}

\begin{corollary} \label{cor:D_i}
	Let $D_1,D_2$ be reduced divisors on $X$ and $\mathcal{D}_i$
	their closures in $\mathcal{X}$. Then
	$$\lim_{T\to \infty} 
	\frac{\sum_{p \leq T} \#(\mathcal{D}_1 \cap \mathcal{D}_2)(\FF_p)/
	\# \mathcal{X}(\FF_p)}
	{(\sum_{p \leq T} \#\mathcal{D}_1(\FF_p)/
	\# \mathcal{X}(\FF_p))^{1/2}
	(\sum_{p \leq T} \#\mathcal{D}_2(\FF_p)/
	\# \mathcal{X}(\FF_p))^{1/2}} = 
	\frac{\langle D_1, D_2 \rangle}{\| D_1\| \| D_2 \|}.$$
\end{corollary} 
\subsection{Proof of Theorem \ref{thm:main}} 
We now take the notation and set up of Theorem \ref{thm:main}.

\subsubsection{Application of Theorem \ref{thm:mainthmrqw}}
 We take $\Omega = \mathcal{X}(\Z) \setminus  \mathcal{D}(\ZZ), h = H$ and 
$$m: \Omega \to \N^n, \quad x \mapsto \left(\prod_{\substack{p \\ x \bmod p \in \mathcal{D}_i(\FF_p)}} p\right)_{i = 1,\dots,n}.$$
That \eqref{def:northcot2qwe} and \eqref{def:omomqwe} hold is clear.
We next show \eqref{def:gdensityqwe} and \eqref{eq:::defmultqwe} using \eqref{eqn:eff_equ}. For this we require the following.

\begin{lemma} \label{lem:points_in_D}
We have 
	$$\#\{ x \in \mathcal{D}(\Z): H(x) \leq B\}  
	= o(\#\{ x \in \mathcal{X}(\Z): H(x) \leq B\}).$$
\end{lemma}
\begin{proof}
	Let $0< \varepsilon < \eta$ and let $p$ be a prime with
	$B^{(\eta-\varepsilon)/M} < p < 2B^{(\eta-\varepsilon)/M}$
	(this exists by Bertrand's  postulate). Then applying
	\eqref{eqn:eff_equ} we obtain
	\begin{align*}
	\frac{\#\{ x \in \mathcal{D}(\Z): H(x) \leq B\}}{\#\{ x \in \mathcal{X}(\Z): H(x) \leq B\}} 
	&\leq \frac{\#\{ x \in \mathcal{X}(\Z): H(x) \leq B, x \bmod p \in \mathcal D(\F_p)\}}{\#\{ x \in \mathcal{X}(\Z): H(x) \leq B\}} \\
	& = \frac{\#\mathcal{D}(\F_p)}{\#\mathcal{X}(\F_p)} + O\left(\frac{p^M}{B^{\eta}}\right)
	\ll \frac{1}{p} + B^{-\varepsilon}
	= o(1)
	\end{align*}
	where the penultimate line is by the Lang--Weil estimates.
\end{proof}

Let $d_1,\dots, d_n$ be square-free and let $d=[d_1,\dots,d_n]$ be their least common multiple. Let
$$\Upsilon_{\b d} = \{ x \in \mathcal{X}(\ZZ/d\ZZ): x \bmod d_i \in \mathcal{D}(\ZZ/d_i\ZZ), i = 1,\dots, n\}.$$
Providing each $d_i$ is coprime to every $p \leq 
 A$, Lemma \ref{lem:points_in_D} and \eqref{eqn:eff_equ} imply
 that
\begin{align*}
\frac{\#\{ x \in \Omega: H(x) \leq B, x \bmod d_i \mid m_i(x), i = 1,\dots, n\}}
{\#\{ x \in \Omega: H(x) \leq B\}}   
 = \frac{\#\Upsilon_{\b d}}{\# \mathcal{X}(\ZZ/d\ZZ)} + O(d^{M} B^{-\eta}).
\end{align*}
Thus \eqref{def:gdensityqwe} holds with
$$g(\mathbf{d}) = \frac{\#\Upsilon_{\b d}}{\# \mathcal{X}(\ZZ/d\ZZ)},$$
where $g$ is supported on vectors $\mathbf{d}$ with square-free entries such that $p \mid d_i \implies p > A$.
To see that $g$ is multiplicative, let
$\gcd(d_1\dots d_n,d_1'\dots d_n') = 1$.
Then
$$g(\mathbf{d}\mathbf{d}') = \frac{\#\Upsilon_{\mathbf{dd'}}}{\# \mathcal{X}(\ZZ/[d_1d_1',\dots,d_nd_n']\ZZ)} = \frac{\#\Upsilon_{\b d}}{\# \mathcal{X}(\ZZ/d\ZZ)} \cdot \frac{\#\Upsilon_{\mathbf{d'}}}{\# \mathcal{X}(\ZZ/d'\ZZ)} = g(\mathbf{d})g(\mathbf{d}')$$
by the Chinese remainder theorem and our coprimality assumption. 
This shows \eqref{eq:::defmultqwe}. 
Next \eqref{eq:asumptnqwe} follows from Proposition \ref{prop:Serre} with $c_i = \langle D_i, D_i \rangle$. To show \eqref{def:northcot2hgf8sqwe} we use the following.

\begin{lemma}
	There exists $c> 0$ such that 
	for all $x \in \mathcal{X}(\ZZ) \setminus \mathcal{D}(\ZZ)$ we have
	$$\prod_{\substack{p \\ x \bmod p \in \mathcal{D}(\FF_p)}} p
	\ll H(x)^{c}.$$
\end{lemma}
\begin{proof}
	Let $\bar{\mathcal{D}}$ be the closure of $D$ in $\PP^d_\Z$. 
	Choose homogeneous polynomials $f_1,\dots,f_r$ over $\ZZ$ which generate
	the ideal of $\bar{\mathcal{D}}$. As $x \notin \mathcal{D}(\ZZ)$, we have 
	$f_i(x) \neq 0$ for some $i$.
	Moreover $x \bmod p \in \mathcal{D}(\FF_p)$  implies that $p \mid f_i(x)$.
	Thus the quantity in question is at most $|f_i(x)| \ll H(x)^{\deg f_i}
	\ll H(x)^{\deg \bar{\mathcal{D}}}$, as required.
\end{proof}
We find that \eqref{def:northcot2hgf8sqwe} holds with $\mathcal{F}(B) \ll B^{c}$. Then \eqref{eq:thebigassumptnqwe}  follows from \eqref{eqn:eff_equ} (see Remark~\ref{rem:zerolevelofdistr}). Finally the limit  \eqref{eqn:gij} exists and equals $\langle D_i, D_j \rangle/\| D_i\| \| D_j \|$ by Corollary \ref{cor:D_i}. Thus all assumptions of Theorem \ref{thm:mainthmrqw} hold and we deduce the first part of Theorem \ref{thm:main}.

\subsubsection{Rank of the matrix}
It remains to prove the final part of Theorem \ref{thm:main}, regarding the formula for the rank of the covariance matrix.  As the $D_i$ are reduced, the matrix $(c_{i,j}/\sqrt{c_{i,i}c_{j,j}})$ is exactly the Gram matrix of the divisors $D_1/\|D_1\|,\dots,D_n/\|D_n\|$ with respect to the inner product on $\Div_\RR X$ defined in \S\ref{sec:inner_product}. However the rank of the Gram matrix is the dimension of the vector subspace of $\Div_\RR X$ generated by the $D_i/\|D_i\|$. But this is also equal to the rank of the subgroup of $\Div X$ generated by the $D_i$. This completes the proof. \qed

\section{Examples} \label{sec:examples}

We now give various examples illustrating our results and use Theorem \ref{thm:main} to prove the special cases stated in the introduction.

\subsection{Complete intersections} \label{sec:Birch}
Here we explain the proof of Theorem \ref{thm:CI}.
We apply
Theorem~\ref{thm:main}
with  $X: f_1=\cdots=f_R=0$,
$d=n-1$ and $D_i=X\cap (x_i=0)$. We take $\mathcal{X}$ to be the model given by taking the closure of $X$ in $\P^d_\Z$.
It suffices to verify~\eqref{eqn:eff_equ}.

Choose $A =A(\b f )> 0$ such that $X$ has good reduction at all primes $p > A$.  Let $Q$ be square-free and supported on primes greater than $A$. Let  $\Upsilon \subset \mathcal{X}(\Z/Q\Z)$ and 
$$N(\Upsilon, B) : = 
\#\{ x \in \mathcal{X}(\ZZ): H(x) \leq B, x \bmod Q \in \Upsilon\}
.$$ 
We first note that the leading term of \eqref{eqn:eff_equ} is known to hold, and follows from equidistribution results of Peyre and standard properties of Tamagawa measures \cite{Pey95, Sal98}.

\begin{lemma}
\label
{lem:sievemethodsrichert} We have 
$$
\lim_{B \to \infty} \frac{N(\Upsilon, B)}
{\#\{ x \in \mathcal{X}(\ZZ): H(x) \leq B\}}
= \frac{\#\Upsilon}{\# \mathcal{X}(\ZZ/Q\ZZ)}.
$$
\end{lemma}
\begin{proof}
By  \cite[Prop.~5.5.3]{Pey95}, Manin's conjecture holds here with respect to arbitrary choices of height function. This implies that the rational points are equidistributed with respect to Peyre's Tamagawa measure \cite[Prop.~3.3]{Pey95}. The measure of the resulting adelic volumes is calculated in \cite[Thm.~2.14(b)]{Sal98} (cf.~\cite[Cor.~2.15]{Sal98}), and gives the stated result.
\end{proof}

It therefore suffices to show that we can obtain an asymptotic formula for $N(\Upsilon, B)$ with an effective error term. 
Denote   the affine cone of $\Upsilon$ by 
$$\widehat{\Upsilon} = 
\l\{ \b y \in (\Z/Q\Z)^n :  \b y \not \equiv \b 0 
\bmod p \, \forall p \mid Q, \, \b y \bmod Q \in \Upsilon
\r \}.$$ We begin with a M\"{o}bius inversion. The key observation in the following lemma is that we may take the M\"{o}bius variable $k$ to be small.

\begin{lemma}
\label
{lem:mobiusinvers}
Fix an arbitrary $\eta_1>0$.
 Then 
for all 
$B \geq 1 $
we have that $N(\Upsilon,B)$ equals 
\[
\frac{1}{2} 
\sum_{\b y \in \widehat{\Upsilon}}
\sum_{\substack{k \in \N \cap [1,B^{\eta_1} ]  \\ \gcd(k,Q) = 1   }} 
 \hspace{-0,3cm}
\mu(k)
\# \l
\{
 \b x \in \l( \Z\cap \l[-
\frac{B}{k}  
,  
\frac{B}{k}  
 \r] \r)^n
\hspace{-0,2cm}
:
  \b f (\b x)=\b 0 ,  \b x \equiv  \frac{  \b y }{k}  \bmod Q 
\r \}
+O_{\b f }
\l(
 Q^n
B^{n-RD-\eta_1 }  
\r)
,\]
 where the implied constant depends at most on $\b f $. 
\end{lemma}
\begin{proof}
 Using M\"{o}bius inversion we see that $N(\Upsilon,B)$ equals 
\begin{align*}& \frac{1}{2}
\#\l \{ \b x \in \Z^n:
\gcd(x_1,\ldots, x_n)=1,
\max_i |x_i| \leq B, \b f (\b x)=\b 0 , \b x \bmod Q \in \widehat{\Upsilon} \r \} 
\\ = & \frac{1}{2}
\sum_{\substack{1\leq k \leq B \\ \gcd(k,Q) = 1}} \mu(k)
\# \{ \b x \in \Z^n:\max_i |x_i| \leq B/k, \b f (\b x)=\b 0 , k\b x \bmod Q \in \widehat{\Upsilon} \}	 
\\ = & \frac{1}{2}
\sum_{\b y \in \widehat{\Upsilon}}\sum_{\substack{k \leq B \\ \gcd(k,Q) = 1}} \mu(k)
\# \{ \b x \in \Z^n:\max_i |x_i| \leq B/k, \b f (\b x)=\b 0 , \b x \equiv k^{-1} \b y  \bmod Q \},
\end{align*}
where the inverse is taken modulo $Q$.
We note that Birch's estimate  \cite[Thm.~1]{Bir62}
ensures that for all $P\geq 1 $ one has
\[
 \# \{ \b x \in \Z^n:\max_i |x_i| \leq P, \b f (\b x)=\b 0 \}
=O_{\b f} 
  \l(
P ^{n-RD}
\r)
.\]Therefore, ignoring the condition $
 \b x \equiv 
 k^{-1} 
\b y \bmod Q$ 
we obtain
\[
\# \{ \b x \in \Z^n:\max_i |x_i| \leq B/k, \b f (\b x)=\b 0 , \b x \equiv k^{-1} \b y  \bmod Q \}
  \ll_{\b f} 
  (B /k )^{n-RD}
.\]
 Noting that our assumptions ensure that $n-RD\geq 2 $ , hence 
  this is $\ll_{\b f}
B^{n-RD} k^{-2}$. 
Using the trivial
bound $ \#\widehat{\Upsilon}
\leq Q^n $ 
we 
therefore
see that 
 for all $L \geq 
1 $ 
we have
\begin{align*}
&
\sum_{\b y \in \widehat{\Upsilon}}
\sum_{\substack{L< k \leq B \\ \gcd(k,Q) = 1   }} 
\mu(k)
\# \{ \b x \in \Z^n:\max_i |x_i| \leq B/k, \b f (\b x)=\b 0 , \b x \equiv k^{-1} \b y  \bmod Q \}
\\
\ll_{\b f }
&
\
Q^n
B^{n-RD}
\sum_{  k> L } 
k^{-2} 
 \ll 
B^{n-RD}
Q^n
L^{-1}
.\end{align*}
Taking $L=B^{\eta_1}$ 
concludes the proof.
\end{proof}
  
We next record 
 the case $\boldsymbol{\nu}=\b 0$
of the work by 
van Ittersum~\cite[Thm. 2.15]{arXiv:1709.05126}.
It gives an effective
error term for the number of integer zeros
 of bounded height on a complete intersection of polynomials which need not be homogeneous.
For a polynomial $g$ let $\widetilde{g}$ denote the homogeneous part of $g$.
\begin{lemma}
[van Ittersum]
\label
{lem:bachcel}
Let $g_1, \ldots, g_R \in \Z[x_1,\ldots, x_n]$  
be arbitrary polynomials of common degree $D$
and assume that 
$\mathfrak{B}(\b g)>2^{D-1}(D-1)R(R+1)$.
Then there exist positive
$M_1,\eta_2$
that depend at most on $\mathfrak{B}(\b g), R$ and $ D$
such that 
\[
\#\l
\{\b z \in \Z^n : \max_{1\leq i \leq n }|z_i| \leq B: \b g(\b z)=\b 0 
\r
\}
=\mathfrak{S}(\b g ) 
J(\widetilde{\b g })
 B^{n-RD}
+O\l( B^{n-RD-\eta_2}
C
 \widetilde{C}^{M_1}
\r)
,\]
where the implied constant 
depends at most on 
$n , D,  R, \mathfrak{B}(\b g)$
and where 
$C$ and $ \widetilde{C}$
   respectively denote 
 the maximum absolute value of the coefficient of all
$g_i$ 
and $\widetilde{g_i}$.
Here $\mathfrak{S}(\b g )$ is the Hardy--Littlewood singular series associated to the system 
$\b g =\b 0$
and $J(\widetilde{\b g })
$ is the Hardy--Littlewood singular integral associated to the system 
$\widetilde{\b g }=\b 0$.
\end{lemma}

Using this, we obtain the following.

\begin{lemma}
 \label
{lem:vanter}
 There exist
$M_2, \eta_3 
>0$ that only depend on $\b f $ 
such that 
for all $\b t \in (\Z /Q \Z)^n$     
the quantity 
$$
\#\{ \b x \in \Z^n:\max_i |x_i| \leq B, \b f (\b x)=\b 0 , \b x \equiv \b t  \bmod{Q} \}
$$
equals 
\[
\sigma_\infty
\l(
\prod_{p\mid Q}
\sigma_p\l( \b t, p^{\nu_p(Q)} \r )
\r)
\l(
\prod_{p\nmid Q}
\sigma_p 
\r)
B^{n-RD}
+O_{\b f}
\l (
 B^{n-RD -\eta_3 } 
Q^{M_2}
\r )
,\]
where the implied constant 
depends at most on $\b f $. Here,
 $\sigma_\infty$ is the standard Hardy--Littlewood 
singular integral associated to the system $\b f =\b 0 $,
\[
 \sigma_p 
:=
\lim_{m\to+\infty }
\frac{
\#\l\{\b x  \in (\Z/p^m\Z)^n: \b f (\b x)\equiv \b 0 \bmod{p^m}  \r\}
}{p^{m(n-R)}},
\]
and for all $e\geq 1 $ and
$\b s \in (\Z/p^e\Z)^n$ we denote 
\[ 
\sigma_p( \b s ,  p^e)
:=
 \lim_{m\to+\infty }
\frac{
\#\l\{\b x \in (\Z/p^m\Z)^n: 
\b f (\b x ) \equiv \b 0 \bmod{p^m} , 
\b x \equiv \b s \bmod {p^{e}} 
\r\}
}{p^{m(n-R)}}
.
\]
 \end{lemma}
\begin{proof}
We first deal with the case $Q>B$. In this instance we plainly have  $$
 \#\{ \b x \in \Z^n:\max_i |x_i| \leq B,  \b x \equiv \b t  \bmod{Q} \} \ll 1,
$$ which is clearly $ \ll B^{n-RD}$ since the  Birch rank assumption implies    $n>RD  $. 
The estimate $\sigma_p (\b s , p^e )\leq \sigma_p$ shows that we always have $$
\sigma_\infty
\l(
\prod_{p\mid Q}
\sigma_p\l( \b t, p^{\nu_p(Q)} \r )
\r)
\l(
\prod_{p\nmid Q}
\sigma_p 
\r)
B^{n-RD}\leq  \sigma_\infty
\l(
\prod_{p }
\sigma_p 
\r)B^{n-RD}
\ll_{\b f } B^{n-RD}.$$  Therefore,  for all $\eta \in (0,1] $ and $M\geq 1 $ one has 
\[B^{n-RD} < B^{n-RD-1 }Q \leq B^{n-RD-\eta } Q^M.\] We are then free to assume that $Q \leq B$ for the rest of the proof.

Without loss of generality we can assume that $\b t \in (\Z \cap [0,Q))^n$.
 We  then   use the change of variables
$\b x =\b t +
Q \b z $ 
to write the counting function in our lemma as 
\[
=\#\left\{ \b z \in \Z^n:\max_i \left|z_i+\frac{t_i}{Q}\right| \leq \frac{B}{Q}, \b f (\b t +Q \b z)=\b 0 \right\}
.\] We now
 apply  Lemma~\ref{lem:bachcel}
 with 
 $g(\b z ):= \b f (\b t +Q \b z)$. 
We
 have 
\[
|z_i| \leq \frac{B}{Q} -1  
\Rightarrow
\left|z_i+\frac{y_i}{Q} \right| \leq \frac{B}{Q}
\Rightarrow
|z_i| \leq \frac{B}{Q} +1  
,\]
therefore, if we let 
$B_1=\frac{B}{Q} -1 $
and 
$B_2=\frac{B}{Q} +1 $
we see that 
\[
\#\{ \b z \in \Z^n:\max_i |z_i| \leq B_j, \b g (\b z)=\b 0 \}
, j=1,2
\]
give lower and upper bounds for the counting function in our lemma, respectively.
We note that $\b f $ and $\b g $ are related via a non-singular
linear change of variables,
hence $\mathfrak{B}(\b g)=\mathfrak{B}(\b f)$.
By Lemma~\ref{lem:bachcel}
and $B_j=B/Q+O(1)$
we therefore
obtain 
\[ 
J(\widetilde{\b g })
\l(\prod_p \tau_p \r)
 (B/Q+O(1) )^{n-RD}
+
O\l( B^{n-rD-\eta_2}
C
 \widetilde{C}^{M_1}
\r)
,\]
where 
\beq
{eq:rushrush}
{
   J(\widetilde{\b g }) 
=\int_{\boldsymbol{\gamma}\in \R^R}
\l(\int_{\boldsymbol{\zeta}\in \R^n}
\exp
\l(2 \pi i 
\sum_{i=1}^R \gamma_i \widetilde{g_i}(\boldsymbol{\zeta})\r)
\mathrm{d}\boldsymbol{\zeta}
\r)
\mathrm{d}\boldsymbol{\gamma}
}
 and  
\beq
{eq:RUSH Jacob's Ladder}
{
\tau_p
:=
\lim_{m\to+\infty }
\frac{
\#\l\{\b z \in (\Z/p^m\Z)^n: \b g (\b z)\equiv \b 0 \bmod{p^m}  \r\}
}{p^{m(n-R)}}
.}
  We note that $\widetilde{ \b g }(\b z )=\b f(Q \b z ) =
Q^D \b f (\b z)$, therefore, 
$\widetilde C \ll_{\b f } Q^D$.
The bound $0\leq 
t_i \leq Q$ and the identity 
$ (Q z_i + t_i )^k =
\sum_{i=0}^k
\l(  { k \choose  j }    Q^j   t_i ^{k-j }  \r)
z_i ^j
$ 
imply 
that $C \ll_{\b f } Q^ D$.
We conclude that 
$ 
C
 \widetilde{C}^{M_1} 
\ll_{\b f }  
  Q^{D(1+M_1)}
$.   Using the bound  $Q \leq B$ we have 
$$(B/Q+O(1) )^{n-RD} =
(B/Q )^{n-RD} +O((B/Q )^{n-RD-1})
,$$ which 
leads to the quantity in our lemma 
being equal to 
 \beq
{eq:rushtime}
{
  J(\widetilde{\b g })
\l(\prod_p \tau_p \r)
 \l(\frac{B}{Q}\r  )^{n-RD}
\hspace{-0,5cm}
+
O\l( 
B^{n-rD-\eta_2}
  Q^{D(1+M_1)}
+\l |
\l(\prod_p   \tau_p \r)
J(\widetilde{\b g })
\r |
(B/Q )^{n-RD-1}
\r)
.} 
 Using
$\widetilde{ \b g }(\boldsymbol{\zeta} )
 = \b f (Q\boldsymbol{\zeta})
=
Q^D 
  \b f ( \boldsymbol{\zeta})
 $ and the change of variables $Q^D \boldsymbol{\gamma}=\boldsymbol{\beta}$
shows that $
 J(\widetilde{\b g }) $ is 
\begin{align*}  
& \int_{\boldsymbol{\gamma}\in \R^R}
\l(\int_{\boldsymbol{\zeta}\in \R^n}
\exp
\l(2 \pi i 
\sum_{i=1}^R (Q^D \gamma_i)  f_i(\boldsymbol{\zeta})\r)
\mathrm{d}\boldsymbol{\zeta}
\r)
\mathrm{d}\boldsymbol{\gamma}
\\ 
=Q^{-RD}
&
\int_{\boldsymbol{\gamma}\in \R^R}
\l(\int_{\boldsymbol{\zeta}\in \R^n}
\exp
\l(2 \pi i 
\sum_{i=1}^R \beta_i   f_i(\boldsymbol{\zeta})\r)
\mathrm{d}\boldsymbol{\zeta}
\r)
\mathrm{d}\boldsymbol{\beta}
. 
\end{align*}
This is clearly $Q^{-RD}
J(\b f )$, in other words, we have seen that 
 $
 J(\widetilde{\b g })  
=
Q^{-RD  }\sigma_\infty$.
This converts~\eqref
{eq:rushtime} into 
  \beq
{eq:your disc is almost full}
{
Q^{-n  }
\sigma_\infty
\l(\prod_p \tau_p \r)
 B  ^{n-RD}
+
O_{\b f }\l( 
B^{n-RD-\eta_2}
Q^{D(1+M_1) }
+
  \l(\prod_p 
\tau_p  \r)
B^{n-RD-1}
Q^{- n  + 1}
\r)
.} For a prime 
  $p\nmid Q $
the change of variables  $(\Z/p^m \Z)^n \to
(\Z/p^m \Z)^n$, 
$\b z \mapsto \b x $ that is 
given by 
$\b x  \equiv \b t +Q \b z \bmod {p^m}$ is invertible modulo $p^m $, therefore,  
the numerator within the limit
in~\eqref
{eq:RUSH Jacob's Ladder}
equals 
\[
 \#\l\{\b z \in (\Z/p^m\Z)^n: \b f  (\b t+ Q \b z )\equiv \b 0 \bmod{p^m}  \r\}
=
\#\l\{\b x \in (\Z/p^m\Z)^n: \b f (\b x)\equiv \b 0 \bmod{p^m}  \r\}
.\] In particular,~\eqref
{eq:RUSH Jacob's Ladder}
agrees with $\sigma_p $.
If $p\mid Q$ a similar argument,
with the map under consideration being 
$\b x \equiv  Q p^{-\nu_p(Q)} 
 \b z 
\bmod{p^k } 
$,
shows that
\[
   \#\l\{\b z \in (\Z/p^m\Z)^n: \b f  (\b t+ Q \b z )\equiv \b 0 \bmod{p^m}  \r\}
=
\#\l\{\b x \in (\Z/p^m\Z)^n: \b f (
\b t +
p^{\nu_p(Q)}
\b x
)\equiv \b 0 \bmod{p^m}  \r\}
. \]
We can clearly 
 rewrite this as 
\[
\sum_{\substack{ \b w \in (\Z/p^m \Z)^n    \\    \b f (\b w ) \equiv \b 0 \bmod{p^m }      }}
\#\{\b x \in (\Z/p^m\Z)^n:    
\b w \equiv  
\b t +
p^{\nu_p(Q)}
\b x
\bmod{p^m } 
\} 
.\]
Now assume that $m > \nu_p(Q)$.
Then 
the inner sum contribution is non-zero only when 
$ \b w \equiv \b t \bmod{ p^{\nu_p(Q)}
 } $. We thus obtain \[
\sum_{\substack{ \b w \in (\Z/p^m \Z)^n   
 \\ 
   \b f (\b w ) \equiv \b 0 \bmod{p^m }    
\\
\b w \equiv \b t \bmod{ p^{\nu_p(Q)  }}
  }}
\#\l
\{\b x \in (\Z/p^m\Z)^n:    
\b x
\equiv 
\frac{ 
\b w  - 
\b t }{
 p^{ \nu_p(Q)}
}
\bmod{p^{m -\nu_p(Q) } } 
\r \} 
.\] The new inner cardinality clearly equals $p^{n \nu_p(Q)}$
since every $x_i \bmod {p^m}$ is uniquely
determined modulo $p^{m-\nu_p(Q)}$.
This gives the
following 
for all $p\mid Q$, 
 \[
\tau_p=
p^{n \nu_p(Q)  }
 \lim_{m\to+\infty }
\frac{
\#\l\{\b x \in (\Z/p^m\Z)^n: \b f (\b x)\equiv \b 0 \bmod{p^m}  , \b x \equiv \b t \bmod{p^{\nu_p(Q)}} \r\}
}{p^{m(n-R)}}
.\]
This is clearly at most 
$
p^{n \nu_p(Q)  }
\sigma_p $, 
therefore, 
 \[
 \prod_{p}
\tau_p 
=
\prod_{p\nmid Q }
 \sigma_p 
\prod_{p\mid Q}
p^{n \nu_p(Q)  }
\sigma_p(  \b  t , p^{\nu_p(Q)} )
\leq 
Q^n \prod_p
\sigma_p 
, \]
which, when injected in~\eqref
{eq:your disc is almost full},
shows that the asymptotic in our lemma holds with an error term 
\[
\ll_{\b f }
 B^{n-RD-\eta_2}
Q^{D(1+M_1)}
+  
B^{n-RD-1}
Q 
\ll
B^{n-RD -\min \{ \eta_2, 1 \}}
Q^{D(1+M_1)}
.\]
Letting  $M_2=\max\{1, {D(1+M_1)}\}$ and $\eta_3:= 
\min\{\eta_2, 1\}$
  concludes the proof of the lemma.\end{proof}
\begin{lemma}
 \label
{lem:vanter2}
 There exist
$M_3, \eta_4 
>0$ that only depend on $\b f $ 
such that 
for all
$Q\in \N$ and $ \Upsilon \subset \mathcal{X}(\Z/Q\Z)
$      
we have 
 \[
N(\Upsilon, B)
=
B ^{n-RD}
\frac{ \sigma_\infty }{2}  
 \l(
\prod_{p\nmid Q} 
\sigma_p
 \l(1-\frac{1}{p^{n-RD}}\r)
 \r)
\sum_{\b y \in \widehat{\Upsilon}}
 \prod_{p\mid Q} 
\sigma_p(\b y, p^{\nu_p(Q)})
+O_{\b f}
\l (
 B^{n-RD -\eta_4 } 
Q^{M_3}
\r )
,\]
where the implied constant 
depends at most on $\b f $.
 \end{lemma}
\begin{proof}
Injecting
Lemma~\ref{lem:vanter}
into Lemma~\ref{lem:mobiusinvers}
  shows
 that 
$N(\Upsilon,B)$ equals 
\begin{align*}
\frac{1}{2} 
\sum_{\b y \in \widehat{\Upsilon}}
\sum_{\substack{k \in \N \cap [1,B^{\eta_1} ]  \\ \gcd(k,Q) = 1   }} 
 &\mu(k)
\l(
\sigma_\infty
\prod_{p\nmid Q} \sigma_p
\prod_{p\mid Q} \sigma_p (\b y/k, p^{\nu_p(Q)})
\l(\frac{B}{k} \r)^{n-RD}
+
O\l(
Q^{M_2}
\l(\frac{B}{k} \r)^{n-RD- \eta_3 }
\r)
\r)  \\
& +O_{\b f }
\l(
 Q^n
B^{n-RD-\eta_1 }  
\r)
.\end{align*}
The fact that $\gcd(k,Q)=1$
shows that for $p\mid Q$ we have 
$
\sigma_p(\b y /k,p^{\nu_p(Q) })
=\sigma_p(\b y  ,p^{\nu_p(Q) })
$. Furthermore, 
the trivial estimate 
$ \#\widehat{\Upsilon}
\leq Q^n $ 
shows that $N(\Upsilon,B)$ equals 
 \[B ^{n-RD}
\frac{ \sigma_\infty }{2}  
 \l(
\prod_{p\nmid Q} \sigma_p
\r)
\l(
\sum_{\substack{k \in \N \cap [1,B^{\eta_1} ]  \\ \gcd(k,Q) = 1   }} 
 \frac{ \mu(k)  }{k^{n-RD} }  
 \r)
\sum_{\b y \in \widehat{\Upsilon}}
 \prod_{p\mid Q} 
\sigma_p(\b y, p^{\nu_p(Q)})
  \]
up to a
term whose modulus is 
\[
\ll
Q^{n+M_2}
\sum_{\substack{k \in \N \cap [1,B^{\eta_1} ]      }} 
  \l(\frac{B}{k} \r)^{n-RD- \eta_3 }
+
  Q^n
B^{n-RD-\eta_1 }  
\ll
Q^{n+M_2}
B^{n-RD-\eta_3+\eta_1}
\] due to the bound  $n-RD-\eta_3 \geq 0$ and 
$\sum_{k\leq B^{\eta_1}} 1 \ll B^{\eta_1}$. 
This is admissible as can be seen by taking 
  $\eta_1=\eta_3/2$ in Lemma~\ref{lem:mobiusinvers}.
The main term contains a sum over $k\in [1,B^{\eta_1}]$ that can be written as 
\[
\prod_{p\nmid Q }\l(1-\frac{1}{p^{n-RD}}\r)
+O\l(\sum_{k>B^{\eta_1}}
\frac{1}{k^2}\r)
=\prod_{p\nmid Q }\l(1-\frac{1}{p^{n-RD}}\r)
+O(B^{-\eta_1})
,\] because our assumptions on the Birch rank
 ensure that $n-RD\geq 2 $.
We plainly have
$
\sigma_p(\b y, p^{\nu_p(Q)})
\leq 
\sigma_p
$
hence 
the     contribution 
of the last term 
$O(B^{-\eta_1})$
is
\[\ll_{\b f }
\l(\prod_{p\nmid Q}\sigma_p\r)
B^{n-RD-\eta_1}
\sum_{\b y \in \widehat{\Upsilon}}
 \prod_{p\mid Q} 
\sigma_p(\b y, p^{\nu_p(Q)})
\leq \l(\prod_{p\nmid Q}\sigma_p\r)
 B^{n-RD-\eta_1}
\sum_{\b y \in \widehat{\Upsilon}}
 \prod_{p\mid Q} 
\sigma_p ,\]
which is $
(
\prod_p \sigma_p )
 B^{n-RD-\eta_1}
\#\Upsilon
\ll_{\b f }
 B^{n-RD-\eta_1}
\#\Upsilon
\leq  B^{n-RD-\eta_1}
Q^n
$.
This is admissible.
The main term is 
\[
B ^{n-RD}
\frac{ \sigma_\infty }{2}  
 \l(
\prod_{p\nmid Q} 
\sigma_p
 \l(1-\frac{1}{p^{n-RD}}\r)
 \r)
\sum_{\b y \in \widehat{\Upsilon}}
 \prod_{p\mid Q} 
\sigma_p(\b y, p^{\nu_p(Q)})
,\] which is as stated
in our lemma.
\end{proof}
We now record the end result of our investigation, which may be 
of independent interest. For completeness, we recall our assumptions.

\begin{proposition}
[Effective equidistribution for Birch systems]
\label
{prop:newbirch}
Assume that $f_1,\ldots, f_R$ are integer   homogeneous
polynomials in $n$ variables, all of the same degree $D$
and   that   
the Birch rank satisfies
	 $\mathfrak{B}(\b f)>2^{D-1}(D-1)R(R+1)$.
	Assume that $f_1=\cdots=f_R=0$ is smooth, that it has a $\Q$-rational point and denote by 
$\mathcal{X}$ 
the model given by taking its
 closure  
in $\P^{n-1}_\Z$. 
Then there exist
positive constants 
 $A, M , \eta$
that only depend on $\b f $
such that for all $Q\in \N$ only divisible by primes $p>A$
and all
$\Upsilon \subset \mathcal{X}(\Z/Q\Z)$, we have
\[\frac{\#\{ x \in \mathcal{X}(\ZZ): H(x) \leq B, x \bmod Q \in \Upsilon\}}
{\#\{ x \in \mathcal{X}(\ZZ): H(x) \leq B\}}=
 \frac{\#\Upsilon}{\# \mathcal{X}(\ZZ/Q\ZZ)}
+O(B^{-\eta}
Q^{M}
)
,\] 
where the implied constant is independent of $B$ and $Q$.
\end{proposition}
\begin{proof}
Using Birch's theorem \cite[Thm.~1]{Bir62} and M\"{o}bius inversion, there exists $\eta_5 >0$
that only depends on $\b f $
such that 
\[
\#\{ x \in \mathcal{X}(\ZZ): H(x) \leq B\}
=
\frac{1}{2\zeta(n-RD)}
\sigma_\infty
\l(\prod_p \sigma_p \r)
B^{n-RD}
+O_{\b f }\l(
B^{n-RD-\eta_5
} 
\r)
.\]
Moreover, our assumptions that $X(\QQ) \neq \emptyset$ and that $X$ is smooth implies that $\sigma_\infty>0$  and $\sigma_p>0$ for all primes $p$. Then 
Lemma~\ref{lem:vanter2}
gives 
 \[
\frac{
N(\Upsilon, B)
}{\#\{ x \in \mathcal{X}(\ZZ): H(x) \leq B\}
}
=  
\sum_{\b y \in \widehat{\Upsilon}}
 \prod_{p\mid Q} 
\frac{
\sigma_p(\b y, p^{\nu_p(Q)})
}{
\sigma_p
 \l(1-\frac{1}{p^{n-RD}}\r)}
+O_{\b f}
\l (
 B^{ -\min\{\eta_4,\eta_5\} } 
Q^{n+M_3}
\r )
\]
with an implied constant depending at most on $\b f $.
The result now follows from 
Lemma~\ref{lem:sievemethodsrichert}. 
\end{proof}
Proposition~\ref
{prop:newbirch}  proves the equidistribution
property~\eqref{eqn:eff_equ}, hence, we may apply Theorem~\ref{thm:main}.
To finish, it suffices to explain why we obtain the identity covariance matrix.

\begin{lemma} \label{lemma:Lefschetz}
	Let $k$ be a field of characteristic zero and 
	$Y \subset \PP^d$ a smooth complete intersection with $\dim Y \geq 3$,
	which is not contained in a hyperplane.
	Then $Y \cap H$ is irreducible for any hyperplane $H \subset \PP^d$.
\end{lemma}
\begin{proof}
	The Lefschetz hyperplane section theorem implies that $\Pic Y \cong \ZZ$
	generated by the hyperplane class \cite[Exposé XII, Cor.~3.7]{SGA2}.
	Thus if $H \cap Y = D_1 + D_2$ for effective divisors $D_1$ and $D_2$,
	we must have $[D_i] = 0$ for some $i$; the result follows.
\end{proof}

Thus the intersections with the coordinate hyperplanes are irreducible, so they contain no common irreducible components.
The result therefore follows from Theorem \ref{thm:main}.
This completes the proof of Theorem \ref{thm:CI}, and Theorem \ref{thm:sexycase} follows immediately.

\subsection{Homogeneous spaces} \label{sec:symmvar}
Counting integral points on homogeneous spaces has a long history and we only mention a few relevant milestones: Duke, Rudnick and Sarnak \cite{DRS} used spectral analysis to deal with a class of (affine) symmetric varieties, 
Gorodnik and Nevo \cite{GorNevo} used the mean ergodic theorem in order to obtain error terms with a power saving; Nevo and Sarnak \cite{NevoSar} have recently applied such counting results to the problem of finding (and estimating the number of) prime or almost-prime points on such varieties.

In this paper we consider the class of symmetric varieties studied by Browning and Gorodnik in \cite{BroGor} and begin by recalling their set-up, which is  more general than  \cite{NevoSar}.
Let $G$ be a connected semisimple algebraic group defined over $\mathbb{Q}$ and let $\iota \colon G \to \mathrm{GL}_n$ be a linear representation defined over $\mathbb{Q}$ with finite kernel. Let $Y \subset \mathbb{A}^n_\Q$ be a subvariety which is left invariant under the action of $G$ via $\iota$. We assume that  $G$ acts transitively on $Y$, so that $Y$ has the form $G/L$ where $L$ is an algebraic subgroup defined over $\mathbb{Q}$. We denote by $Y(\Z) = \mathcal{Y}(\Z)$, where $\mathcal{Y} \subset \A^n_\Z$ is the model given by the closure of $Y$ in $\A^n_\Z$. We assume that $Y(\mathbb{Z}) \ne \emptyset$.
Moreover, the following assumptions are made:
\begin{enumerate}
\item $L$ is a symmetric subgroup of $G$, meaning the Lie algebra of $L$ is the fixed locus of a non-trivial involution defined over $\mathbb{Q}$; 
\item the connected component of $L$ has no non-trivial $\mathbb{Q}$-rational characters; 
\item the group $G$ is $\mathbb{Q}$-simple; 
\item the group $G(\mathbb{R})$ is connected and has no compact factors.
\end{enumerate}
This is the class of symmetric varieties $Y$ which we shall be interested in. For $\b y \in Y(\Z)$,  we define its height by $H(\b y) = \max_{i \in \{1, \ldots, n\}} |y_i|$. We use the following result, stated in \cite[Prop.~3.1]{BroGor}.

\begin{proposition}\label{prop:BG} There exists $\delta > 0$ such that for every $\ell \in \mathbb{N}$ and every $\b  \xi \in Y(\Z/\ell \Z)$ we have
\begin{align*}
\#\{ \b  y \in Y(\Z) : H(\b y) \le B, \b y \equiv \b \xi \bmod \ell \} \\
= \mu_\infty(Y; B) \prod_{p \text{ prime}} \hat{\mu}_p(Y; \b \xi, \ell) &+ O(\ell^{\dim(L)+2\dim(G)} \mu_\infty(Y;B)^{1-\delta})
\end{align*}
as $B \to \infty$, where 
$$\hat \mu_p(Y; \b  \xi, \ell)=\lim_{t \to \infty} p^{-t \dim(Y)} \#\{ \b  y \in Y(\Z/p^t\Z) : \b  y \equiv \b  \xi \bmod{p^{v_p(\ell)}} \}$$
is the $p$-adic density and $\mu_\infty(Y;B)$ is the real density, as defined in \cite[(1.6)]{BroGor}.
\end{proposition}
It follows from this and Hensel's lemma that 
\begin{align*}
&\#\{ \b y \in Y(\Z) : H(\b y) \le B, \b y \equiv \b \xi \bmod \ell \} \\
&= \frac{\#\{ \b y \in Y(\Z/\ell \Z) : \b y = \b \xi \}}{\#Y(\Z/\ell\Z)} N(Y,B) + O\left(\ell^{\dim(L)+2\dim(G)} N(Y,B)^{1-\delta} \right),
\end{align*}
where $
N(Y,B)(\Z) = \#\{ \b x \in Y(\mathbb{Z}) : H(\b x) \le B \}. $
The effective equidistribution property \eqref{eqn:eff_equ} is therefore easily seen to hold in this case. Theorem \ref{thm:main} thus shows the following.
\begin{theorem} \label{thm:symvar}
Let $Y \subset \A^n$ be a symmetric variety in the class described above and let $\Omega_B=\{\b y \in Y(\Z): H(\b y) \leq B , y_1\cdots y_n \neq 0 \}$ be equipped with the uniform probability measure. Then as $B \to \infty$ the random vectors 
\[\Omega_B \to \R^n,  \quad
	\b y \mapsto \left( \frac {\omega(y_1) - c_{1,1} \log \log B}{\sqrt{c_{1,1} \log \log B}}, \ldots, \frac {\omega(y_n) - c_{n,n} \log \log B}{\sqrt{c_{n,n} \log \log B}} \right)
 \]
 converge in distribution to the central multivariate distribution with covariance matrix whose $(i,j)$-entry is $c_{i,j}$, the number of common irreducible components of $y_i = 0$ and $y_j = 0$ in $Y$.
\end{theorem}
Of course Theorem \ref{thm:main} also gives a version for general divisors $D_i$.
We now explain how \autoref{thm:indQF} and \autoref{thm:deteqn} are corollaries of the above theorem and why the covariance matrix is the identity in these examples.

\begin{proof}[Proof of Theorem \ref{thm:indQF}]
That the varieties in \autoref{thm:indQF} fall under the setting of this section is explained in \cite[Rem.~1.3]{BroGor}.
The conclusion now follows immediately from Theorem \ref{thm:symvar}, as it is easily checked that the intersection with each coordinate hyperplane is irreducible (for $n=3$ this follows from our assumption that $-k\disc(Q)$ is not a perfect square). 
\end{proof}

\begin{proof}[Proof of Theorem \ref{thm:deteqn}]
That the varieties in \autoref{thm:deteqn} fall under the setting of this section can be seen as follows. Let $G = \mathrm{SL}_n \times \mathrm{SL}_n$ act on the space $\mathcal{M}_n$ of $n \times n$ matrices by mapping $M \in \mathcal{M}_n$ to $g^{-1} M h$, for  $(g, h) \in G$. Then $V_{n,k} = G/L$, with $L = \mathrm{SL}_n$ being diagonally embedded in $G$.
Here again the conclusion easily follows from the fact that the intersection with each coordinate hyperplane is irreducible. This can be proved, for example, by applying a suitable version of the Lefschetz hyperplane section to the intersection of a hyperplane with the projectivised hypersurface $\det(M) = kz^n$. 
\end{proof}

\begin{remark}
	Let us note that for general choices of symmetric varieties $Y \subset \mathbb{A}^n$ in Theorem~\ref{thm:symvar}, one can  obtain non-identity covariance matrices. For example, let $\sigma_d:\SL_n \to \GL_N$ be the $d$th symmetric power representation and take $G = \SL_n \times \SL_n$ with the representation $G \to \GL_N, (g,h) \mapsto \sigma_d(g)^{-1} \sigma_d(h)$. Then $Y$, given by the orbit of the identity matrix, has the stated property for $d > 1$  (this is a variant of the construction in Remark \ref{rem:d-uple}).
\end{remark}

\subsection{Conics}
We now prove our results on conics from \S \ref{sec:conics}. 
\subsubsection{Proof of Theorem \ref{thm:conics}}
Any smooth conic with a rational point is isomorphic to the projective line. The effective equidistribution property \eqref{eqn:eff_equ} is known to hold for the projective line \cite[Prop.~2.1]{LS19}. The result \emph{loc.~cit.} is proved for the standard height on $\P^1_\Q$, but a minor modification shows that property  \eqref{eqn:eff_equ} in fact holds for more general choices of height function, for some choice of $M$ and $\eta$, which in particular shows that the hypotheses of Theorem \ref{thm:main} hold in this case. This therefore immediately gives the result. \qed

\subsubsection{Proof of Theorem \ref{thm:conic_singular}}
Let $C$ be as in Theorem \ref{thm:conic_singular}.
Let $D'_i: x_i = 0$ and $D_i = D'_{i,\mathrm{red}}$.
As the covariance matrix is singular, Theorem \ref{thm:main}
shows that there is
a linear relation between the divisors $D_i$
in $\Div C$.
But we have $D_i' = b_iD_i$ for some $b_i \in \{1,2\}$.
Thus this also gives a relation
$$a_0D_0' + a_1D_1' + a_2D_2' = 0$$
between the $D_i'$. We take the minimal such relation, so
that $\gcd(a_0,a_1,a_2) =1$.
Moreover, as $\deg D_i' = 2$, we find that
$a_0 + a_1 + a_2 = 0$. 
Changing signs as required and permuting coordinates, we obtain the relation
$$x_1^{a_1}x_2^{a_2} = cx_0^{a_1+a_2}$$
in the homogeneous coordinate ring of $C$, for some $c \in \QQ$.
But the only relation in this ring is the equation of the
conic $C: Q(x_0,x_1,x_2) = 0$, hence
$Q \mid x_1^{a_1}x_2^{a_2} - cx_0^{a_1+a_2}$.
But $\gcd(a_1,a_2) = 1$, implies that this polynomial is irreducible,
hence we must have $Q = c'(x_1^{a_1}x_2^{a_2} - cx_0^{a_1+a_2})$ 
for some $c' \in \QQ$. As $\deg Q = 2$ we have $a_1 = a_2 = 1$,
as required. \qed

\subsection{A cubic surface}
Consider the cubic surface
$$X: \quad x_1x_2x_3 = x_0^3 \qquad \subset \PP^3_\Q.$$
With respect to the coordinate hyperplanes $x_i = 0$, we conjecture that an analogue
of Theorem \ref{thm:main} holds with covariance matrix 
\begin{equation} \label{eqn:conj}
\begin{pmatrix}
	1 & \sqrt{5}/3 & \sqrt{5}/3 & \sqrt{5}/3 \\
	\sqrt{5}/3 & 1 & 2/5 & 2/5\\
	\sqrt{5}/3 & 2/5 & 1 & 2/5\\
	\sqrt{5}/3 & 2/5 & 2/5 & 1
\end{pmatrix}.
\end{equation}
\iffalse
\begin{equation} \label{eqn:conj}
\begin{pmatrix}
	1 & \displaystyle{\frac{\sqrt{5}}{3}} & \displaystyle{\frac{2}{3\sqrt{5}}} & \displaystyle{\frac{2}{3\sqrt{5}}}\\
	\displaystyle{\frac{2}{3\sqrt{5}}} & 1 & \frac{2}{5} & 1\\
	\displaystyle{\frac{2}{3\sqrt{5}}} & 0 & 1 & 1\\
	\displaystyle{\frac{2}{3\sqrt{5}}} & 1 & 1 & 1
\end{pmatrix}.
\end{equation}
\fi
Let us explain how we obtained this. First $X$ is singular, and the counting problem should really take place on the minimal desingularisation $\widetilde{X}$ of $X$. We then naively apply the formula from Theorem \ref{thm:main} with respect to the divisors $D_i$ on $\widetilde{X}$ given by the pull backs of the (reduced) hyperplanes $H_i: x_i = 0$. For $i \neq 0$ the $H_i$ are the lines $x_i= x_0= 0$, whereas $H_0$ is the union of these three lines. Any two lines meet in a singular point of type $A_2$, and these are all the singular points. The singularities are resolved by blowing-up twice, which introduces $2$ new exceptional curves. As any line contains two singular points, a calculation using the above considerations shows that
$$c_{i,i}=
\begin{cases}
	5, & i \in \{1,2,3\}, \\
	9, & i =0.
\end{cases}
\quad  \quad
c_{i,j} =
\begin{cases}
	2, & i \neq j \in \{1,2,3\}, \\
	5, & i = 0, j \in \{1,2,3\}.
\end{cases}
$$
The formula \eqref{eqn:conj} now easily follows. Note that this matrix is singular,  due to the obvious relation $H_0=H_1 + H_2 + H_3$.

How would one go about proving this? Firstly, it follows from \cite[\S3.10]{CLT10} that the rational points on $\widetilde{X}$ are equidistributed, which gives the main term in \eqref{eqn:eff_equ}. 
One can prove that
\begin{equation} \label{eqn:logs}
\#\{ x \in \widetilde{\mathcal{X}}(\ZZ): H(x) \leq B, x \bmod Q \in \Upsilon\} = B P_\Upsilon(\log B) + O_\Upsilon(B^{1-\eta}),
\end{equation}
where $\eta>0$ and 
$P_\Upsilon$ is a polynomial of degree $6$ whose coefficients depend on $\Upsilon$. 
By equidistribution one understands  the leading coefficient of $P_\Upsilon$; the challenge lies with controlling the dependence on the lower order terms
of $P_\Upsilon$ and whether an asymptotic formula of the shape \eqref{eqn:logs}, with powers of $\log B$ appearing, can be used to obtain an Erd\H{o}s--Kac law.


\begin{thebibliography}{99}
 



 \bibitem{MR0466055}
P. Billingsley,
The probability theory of additive arithmetic functions.
 {\em Ann. Probability}, 
{\bf 2}, (1974), 749--791.

\bibitem
{MR1324786}
\bysame,
 {\em Probability and measure}.
Third Edition,
A Wiley-Interscience Publication, 
John Wiley \& Sons, Inc., New York, 1995.

\bibitem
{MR1700749}
\bysame,
 {\em Convergence of probability measures}.
Second edition,
Wiley Series in Probability and Statistics, 
John Wiley \& Sons, Inc., New York, 1999.

\bibitem{Bir62}
B. J. Birch,  
Forms in many variables.
{\em Proc. Roy. Soc. London Ser. A} {\bf 265} (1961/62), 245--263. 

 

\bibitem{BorRud}
M. Borovoi, Z. Rudnick,
Hardy--Littlewood varieties and semisimple groups.
{\em Invent. Math.}
{\bf 119}, (1995), 37--66.


\bibitem{BGS10}
J. Bourgain, A. Gamburd, P. Sarnak,
Affine linear sieve, expanders, and sum-product.
{\em Invent. Math.}
{\bf 179}, (2010), 559--644.

\bibitem{BroGor}
T.D. Browning, A. Gorodnik,
Power-free values of polynomials on symmetric varieties.
 {\em Proc. London Math. Soc.}
{\bf 114}, (2017), 1044--1080.


\bibitem{CLT10}
A. Chambert-Loir, Y. Tschinkel, 
Integral points of bounded height on toric varieties.
\texttt{arxiv:1006.3345}.

\bibitem{DRS}
W. Duke, Z.  Rudnick, P. Sarnak,
Density of integer points on affine homogeneous varieties.
{\em Duke Math. J.}
{\bf 71}, (1993), 143--179.

\bibitem{Dur19}
R. Durret,
\emph{Probability: Theory and Examples},
Fifth edition. Cambridge Series in Statistical and Probabilistic Mathematics, {\bf49}. Cambridge University Press, Cambridge, 2019.

\bibitem
{delbz18}
D.~El-Baz,
An analogue of the {E}rd{\H{o}s}-{K}ac theorem for the special linear group over the integers.
 \emph{Acta Arith}, to appear.
 
 \bibitem{EK40}
P. Erd\H{o}s, M. Kac,  The Gaussian law of errors in the theory of additive number theoretic functions. \emph{Amer. J. Math.}, \textbf{62}, (1940), 738--742.

\bibitem{Erd46}
P. Erd\H{o}s,
 On the distribution function of additive functions. \emph{Ann. of Math.}, \textbf{47}, (1946), 1--20.
 

\bibitem{GorNevo}
A. Gorodnik, A. Nevo,
Quantitative ergodic theorems and their number-theoretic applications.
\emph{Bull. Amer. Math. Soc. (N.S.)},
\textbf{52}, (2015), 65--113.

\bibitem{grs}
A. Granville, K. Soundararajan,
Sieving and the {E}rd{\H{o}s}-{K}ac theorem.
{\em Equidistribution in number theory, an introduction}, NATO Sci. Ser. II Math. Phys. Chem.,
Springer, Dordrecht,
{\bf 237}, (2007), 15--27.

\bibitem{SGA2}
A. Grothendieck,
Cohomologie locale des faisceaux coh\'erents et th\'eor\`emes de Lefschetz locaux et globaux (SGA 2).
 Documents Math\'ematiques (Paris), 4. Société Mathématique de France, Paris, 2005.

\bibitem
{Hal56}
H. Halberstam,
On the distribution of additive number-theoretic functions. {II}.
{\em J. London Math. Soc.}, 
{\bf 31}, (1956), 1--14.

\bibitem
{MR73627}
\bysame,
On the distribution of additive number-theoretic functions. {III}.
{\em J. London Math. Soc.}, 
{\bf 31}, (1956), 14--27.


 
\bibitem{LW54}
S.~Lang, A.~Weil, Number of points of varieties in finite fields.
 \emph{Amer. J. Math.} \textbf{76} (1954), 819--827.
 
\bibitem{Lev49}
W. LeVeque, On the size of certain number-theoretic functions. \emph{Trans. Amer. Math. Soc.} \textbf{66} (1949), 440--463.

\bibitem{LS19}
D.~Loughran, E.~Sofos, 
An {E}rd{\H{o}s}-{K}ac law for local solubility in families of varieties.
 \emph{Selecta Math.}, to appear.

\bibitem{NevoSar}
A. Nevo, P. Sarnak,
Prime and almost prime integral points on principal homogeneous spaces.
{\em Acta Math.}
{\bf 205}, (2010), 361--402.

\bibitem{Pey95}
E.~Peyre, \emph{Hauteurs et mesures de {T}amagawa sur les vari\'et\'es de
  {F}ano}, Duke Math. J. \textbf{79} (1995), no.~1, 101--218.
  
  \bibitem{Sal98}
P. Salberger,
{Tamagawa measures on universal torsors and points of bounded height on Fano varieties}.
\emph{Ast\'erisque} {\bf 251} (1998), 91--258.

\bibitem{Ser12}
{J.-P. Serre}, \emph{Lectures on $N_X(p)$}.
Chapman \& Hall/CRC Research Notes in Mathematics, 11. CRC Press, Boca Raton, FL, 2012.

 


\bibitem{Tan55}
M. Tanaka,
On the number of prime factors of integers.
{\em Jap. J. Math.}, 
{\bf 25 }(1956), no. 8, 1--20. 



 \bibitem{MR1342300}
G. Tenenbaum, 
{\em Introduction to analytic and probabilistic number theory.}
Cambridge Studies in Advanced Mathematics,\  {\bf 46},
Cambridge University Press, Cambridge, 1995.

 



\bibitem
{arXiv:1709.05126}
J.-W. M. van Ittersum,
Quantitative results on Diophantine equations in many variables.
{\em Acta Arithmetica},
to appear,
{\em arXiv:1709.05126},
 (2017).




\bibitem{Xio09}
M. Xiong,
The {E}rd{\H{o}s}-{K}ac theorem for polynomials of several variables.
{\em Proc. Amer. Math. Soc.}, 
{\bf 137 }(2009), no. 8, 2601--2608. 

 
\end{thebibliography}
\end{document}